\documentclass[12pt]{article}
\usepackage{amsmath,enumerate,amsfonts,amssymb,color,graphicx,amsthm}

\usepackage[colorlinks=true, allcolors=black]{hyperref}
\usepackage{todonotes}
\usepackage[normalem]{ulem}
\numberwithin{equation}{section}
\usepackage{cite}
\usepackage{mathtools}
\usepackage[nottoc,notlot,notlof]{tocbibind}
\usepackage[toc,page]{appendix}

\usepackage{lipsum}

\def\XXint#1#2#3{{\setbox0=\hbox{$#1{#2#3}{\int}$ }
		\vcenter{\hbox{$#2#3$ }}\kern-.6\wd0}}

\newlength{\leftstackrelawd}
\newlength{\leftstackrelbwd}
\def\leftstackrel#1#2{\settowidth{\leftstackrelawd}%
	{${{}^{#1}}$}\settowidth{\leftstackrelbwd}{$#2$}%
	\addtolength{\leftstackrelawd}{-\leftstackrelbwd}%
	\leavevmode\ifthenelse{\lengthtest{\leftstackrelawd>0pt}}%
	{\kern-.5\leftstackrelawd}{}\mathrel{\mathop{#2}\limits^{#1}}}

\theoremstyle{plain}
\newtheorem{thm}{Theorem}[section]
\newtheorem{lem}[thm]{Lemma}
\newtheorem{cor}[thm]{Corollary}
\newtheorem{prop}[thm]{Proposition}
\newtheorem*{thm*}{Theorem}

\theoremstyle{definition}
\newtheorem{defn}[thm]{Definition}

\newtheorem{rmk}[thm]{Remark}
\newtheorem{?}[thm]{Problem}

\newcommand{\ep}{\varepsilon}
\renewcommand{\phi}{\varphi}
\renewcommand{\epsilon}{\varepsilon}
\makeatletter
\def\@cite#1#2{[\textbf{#1\if@tempswa , #2\fi}]}
\def\@biblabel#1{[\textbf{#1}]}
\makeatother

\makeatletter
\newcommand*{\defeq}{\mathrel{\rlap{%
			\raisebox{0.3ex}{$\m@th\cdot$}}%
		\raisebox{-0.3ex}{$\m@th\cdot$}}%
	=}
\makeatother

\makeatletter
\newcommand*{\eqdef}{=\mathrel{\rlap{%
			\raisebox{0.3ex}{$\m@th\cdot$}}%
		\raisebox{-0.3ex}{$\m@th\cdot$}}%
}
\makeatother

\newcounter{marnote}

\makeatletter
\def\underbracex#1#2{\mathop{\vtop{\m@th\ialign{##\crcr
				$\hfil\displaystyle{#2}\hfil$\crcr
				\noalign{\kern3\p@\nointerlineskip}%
				#1\crcr\noalign{\kern3\p@}}}}\limits}

\def\upbracefilla{$\m@th \setbox\z@\hbox{$\braceld$}%
	\bracelu\leaders\vrule \@height\ht\z@ \@depth\z@\hfill 
	\kern\p@\vrule \@width\p@\kern\p@\vrule \@width\p@\kern\p@\vrule \@width\p@
	$}

\def\upbracefillb{$\m@th \setbox\z@\hbox{$\braceld$}%
	\vrule \@width\p@\kern\p@\vrule \@width\p@\kern\p@\vrule \@width\p@\kern\p@
	\leaders\vrule \@height\ht\z@ \@depth\z@\hfill\bracerd
	\braceld\leaders\vrule \@height\ht\z@ \@depth\z@\hfill
	\kern\p@\vrule \@width\p@\kern\p@\vrule \@width\p@\kern\p@\vrule \@width\p@
	$}

\def\upbracefillc{$\m@th \setbox\z@\hbox{$\braceld$}%
	\vrule \@width\p@\kern\p@\vrule \@width\p@\kern\p@\vrule \@width\p@\kern\p@
	\leaders\vrule \@height\ht\z@ \@depth\z@\hfill
	\kern\p@\vrule \@width\p@\kern\p@\vrule \@width\p@\kern\p@\vrule \@width\p@
	$}

\def\upbracefilld{$\m@th \setbox\z@\hbox{$\braceld$}%
	\vrule \@width\p@\kern\p@\vrule \@width\p@\kern\p@\vrule \@width\p@\kern\p@
	\leaders\vrule \@height\ht\z@ \@depth\z@\hfill\braceru$}

\def\upbracefillbd{$\m@th \setbox\z@\hbox{$\braceld$}%
	\vrule \@width\p@\kern\p@\vrule \@width\p@\kern\p@\vrule \@width\p@\kern\p@
	\bracerd\braceld
	\leaders\vrule \@height\ht\z@ \@depth\z@\hfill\braceru$}

\makeatother
\date{}

\begin{document}

	\title{The fully nonlinear Loewner-Nirenberg problem: Liouville theorems and counterexamples to local boundary estimates}
	\author{Jonah A. J. Duncan\footnote{Department of Mathematics, University College London, 25 Gordon Street, London, WC1H 0AY, UK. Email: jonah.duncan@ucl.ac.uk. Supported by the Additional Funding Programme for Mathematical Sciences, delivered by EPSRC (EP/V521917/1) and the Heilbronn Institute for Mathematical Research.} ~and Luc Nguyen\footnote{Mathematical Institute and St Edmund Hall, University of Oxford, Andrew Wiles Building, Radcliffe Observatory Quarter, Woodstock Road, OX2 6GG, UK. Email: luc.nguyen@maths.ox.ac.uk}}
	\maketitle

	\begin{abstract}
		In this paper we give a complete classification of positive viscosity solutions $w$ to conformally invariant equations of the form
		\begin{align}\label{ab}\tag{$*$}
		\begin{cases}
		f(\lambda(-A_w)) = \frac{1}{2}, \quad \lambda(-A_w)\in\Gamma & \text{in }\mathbb{R}_+^n \\
		w = 0 & \text{on }\partial\mathbb{R}_+^n,
		\end{cases}
		\end{align}
		where $A_w$ is the Schouten tensor of the metric $g_w = w^{-2}|dx|^2$, $\Gamma\subset\mathbb{R}^n$ is a symmetric convex cone and $f$ is an associated defining function satisfying standard assumptions. Solutions to \eqref{ab} yield metrics $g_w$ of negative curvature-type which are locally complete near $\partial\mathbb{R}_+^n$. In particular, when $(f,\Gamma) = (\sigma_1,\Gamma_1^+)$, \eqref{ab} is the Loewner-Nirenberg problem in the upper half-space. \medskip 
		
		More precisely, let $\mu_\Gamma^+$ denote the unique constant satisfying $(-\mu_\Gamma^+, 1,\dots,1)\in\partial\Gamma$. We show that when $\mu_\Gamma^+ >1$ (e.g.~when $\Gamma = \Gamma_k^+$ for $k<\frac{n}{2}$),  the hyperbolic solution $w^{(0)}(x) \defeq x_n$ is the unique solution to \eqref{ab}. More surprisingly, we show that when $\mu_\Gamma^+ \leq 1$ (e.g.~when $\Gamma = \Gamma_k^+$ for $k\geq \frac{n}{2}$), the solution set consists of a monotonically increasing one-parameter family $\{w^{(a)}(x_n)\}_{a\geq 0}$, of which the hyperbolic solution $w^{(0)}$ is the \textit{minimal} solution. In either case, solutions of \eqref{ab} are functions of $x_n$. Our proof involves a novel application of the method of moving spheres for which we must establish new estimates and regularity near $\partial\mathbb{R}_+^n$, followed by a delicate ODE analysis. As an application, we give counterexamples to local boundary $C^0$ estimates on solutions to the fully nonlinear Loewner-Nirenberg problem when $\mu_\Gamma^+ \leq 1$.\bigskip 
		~
	\end{abstract}

	\setcounter{tocdepth}{2}
	\tableofcontents

\section{Introduction}

A central problem in analysis is to characterise solutions to PDEs on model spaces such as $\mathbb{R}^n$ or $\mathbb{R}_+^n = \{x\in\mathbb{R}^n: x_n >0\}$, possibly subject to growth, sign and/or boundary conditions as appropriate. The prototypical result of this type is due to Liouville \cite{Liou1853}, who showed that a positive harmonic function defined on $\mathbb{R}^n$ must be constant. Such Liouville-type theorems are closely related to the validity of estimates on solutions to PDEs on more general domains or Riemannian manifolds, and are consequently of importance in the theory for such equations. \medskip

 In this paper we characterise solutions to a very general class of conformally invariant PDEs of negative curvature type on $\mathbb{R}^n_+$ ($n\geq 3$), for which the hyperbolic metric serves as the model solution. We show that whether the hyperbolic metric is the unique solution is completely determined by a single parameter $\mu_\Gamma^+$ (defined in \eqref{220} below), and in the case of non-uniqueness we classify all other solutions. As a consequence, we show that local boundary $C^0$ estimates fail for the fully nonlinear Loewner-Nirenberg problem when $\mu_\Gamma^+ \leq 1$. This is in contrast to the case $\mu_\Gamma^+>1$, where such estimates were established by the authors in \cite{DN23} and used to obtain existence of solutions to the fully nonlinear Loewner-Nirenberg problem on all Riemannian compact manifolds with boundary.\footnote{In spite of the failure of such estimates when $\mu_\Gamma^+ \leq 1$, we show in \cite{DN25b} that the fully nonlinear Loewner-Nirenberg problem on a compact Riemannian manifold with boundary always admits a solution if $\mu_\Gamma^+>1-\delta$, where $\delta>0$ is a constant depending on certain global geometric data. It remains an interesting open problem to determine whether solutions exist for smaller values of $\mu_\Gamma^+$.}\medskip

Let us now describe the set-up for our work in more detail. Given a positive $C^2$ function $w$ defined on a subset of $\mathbb{R}^n$ ($n\geq 3$), we denote the conformally flat metric $g_w = w^{-2}|dx|^2$ and denote by $A_w$ the $(1,1)$-Schouten tensor of $g_w$:
\begin{align}\label{305}
A_w = w\nabla^2 w - \frac{1}{2}|\nabla w|^2 I,
\end{align}
where $\nabla^2 w$ is the Euclidean Hessian matrix of $w$ and $I$ is the $n\times n$ identity matrix. The quantity $w^{-1}A_w$ is sometimes referred to as the \textit{conformal Hessian} of $w$. We study the Dirichlet boundary value problem
\begin{align}\label{15}
\begin{cases}
f(\lambda(-A_w)) = \frac{1}{2}, \quad \lambda(-A_w)\in\Gamma & \text{in }\mathbb{R}_+^n \\
w = 0 & \text{on }\partial\mathbb{R}_+^n,
\end{cases}
\end{align}
where $f$ and $\Gamma$ are assumed to satisfy the following properties: 
\begin{align}
& \Gamma\subset\mathbb{R}^n\text{ is an open, convex, connected symmetric cone with vertex at 0}, \label{21'} \\
& \Gamma_n^+ = \{\lambda\in\mathbb{R}^n: \lambda_i > 0 ~\forall ~1\leq i \leq n\} \subseteq \Gamma \subseteq \Gamma_1^+ =  \{\lambda\in\mathbb{R}^n : \lambda_1+\dots+\lambda_n > 0\}, \label{22'} \\
& f\in C^\infty(\Gamma)\cap C^0(\overline{\Gamma}) \text{ is homogeneous of degree one and symmetric in the }\lambda_i, \label{23'}  \\
& f>0 \text{ in }\Gamma, \quad f = 0 \text{ on }\partial\Gamma, \quad f_{\lambda_i} >0 \text{ in } \Gamma \text{ for }1 \leq i \leq n. \label{24'}
\end{align}
\noindent When convenient, we may also assume without loss of generality that $f$ is normalised so that 
\begin{align}\label{25'}
f(e) = 1, \quad e = (1,\dots,1),
\end{align}
and at various points we will also assume that 
\begin{align}\label{26'}
f \text{ is concave}. 
\end{align}
In fact, we will only ever use a much weaker condition than \eqref{26'} -- see \eqref{23'''} in Section \ref{s3}. We stress that we will not make any assumptions on the behaviour of $w$ at infinity in our study of \eqref{15}. \medskip

In general, the equation in \eqref{15} is a fully nonlinear, non-uniformly elliptic equation in $w$, similar in structure to the Hessian equations that have received significant attention since the seminal work of Caffarelli, Nirenberg \& Spruck \cite{CNS3}. Typical examples of $(f,\Gamma)$ satisfying \eqref{21'}--\eqref{26'} are given by $(c\sigma_k^{1/k}, \Gamma_k^+)$ for $1\leq k \leq n$, where $c = c_{n,k} =\binom{n}{k}^{-1/k}$, $\sigma_k:\mathbb{R}^n\rightarrow\mathbb{R}$ is the $k$'th elementary symmetric polynomial and $\Gamma_k^+$ is the G{\aa}rding cone:
\begin{align*}
\sigma_k(\lambda_1,\dots,\lambda_n) = \sum_{1\leq i_1 < \dots< i_k \leq n} \lambda_{i_1}\dots\lambda_{i_k}, \qquad \Gamma_k^+ = \{\lambda \in\mathbb{R}^n : \sigma_j(\lambda)>0 \text{ for }1 \leq j \leq k\}.
\end{align*}
In particular, when $(f,\Gamma) = (\frac{1}{n}\sigma_1, \Gamma_1^+)$, after making the substitution $u^{\frac{4}{n-2}} = w^{-2}$, \eqref{15} is equivalent to the semilinear problem
\begin{align}\label{303}
\begin{cases}
\Delta u = \frac{1}{4}n(n-2)u^{\frac{n+2}{n-2}} & \text{in }\mathbb{R}_+^n \\
u(x) \rightarrow +\infty \text{ as }\operatorname{dist}(x,K)\rightarrow 0 & \text{for any fixed compact set }K\subset\partial\mathbb{R}_+^n. 
\end{cases}
\end{align}
The equation in \eqref{303} is the Yamabe equation in the case of negative scalar curvature and has been studied extensively in the literature -- see later in the introduction for a partial list of references. The equation in \eqref{15} is therefore a fully nonlinear Yamabe-type equation, and the boundary data corresponds to local completeness of the metric $g_w$ near $\partial\mathbb{R}_+^n$ (see after Theorem \ref{B} for the definition). \medskip

An important property of \eqref{15} that will be used in our analysis is its conformal invariance (see Section \ref{sym} for the definition). In \cite{LL03}, Li \& Li characterised conformally invariant second order differential operators as those of the form $F[w] = f(\lambda(A_w))$, where $f:\mathbb{R}^n\rightarrow\mathbb{R}$ is a symmetric function and $\lambda(A_w)$ denotes the eigenvalues of $A_w$. See also the work of Li, Lu \& Lu \cite{LLL21} for the case $n=2$, and the work of Li \& Nguyen \cite{LN09b} for a similar characterisation of conformally invariant first order boundary operators involving normal derivatives.\medskip

It is well-known that up to scaling, the hyperbolic metric $x_n^{-2}|dx|^2$ is the unique complete metric with constant negative sectional curvature on $\mathbb{R}^n_+$. In particular, $w^{(0)}(x) \defeq x_n$ is always a solution to \eqref{15} -- for this reason we refer to $w^{(0)}$ as the \textit{hyperbolic solution}. A fundamental question is the canonicality of the hyperbolic metric in the study of \eqref{15} and more generally the fully nonlinear Loewner-Nirenberg and Yamabe problems of negative curvature type. In particular, it is important to determine whether the hyperbolic solution is the unique solution to \eqref{15}. One main purpose of this paper is to provide a complete answer to this question. Recall that the quantity $\mu_\Gamma^+$, introduced by Li \& Nguyen in \cite{LN14b}, is defined by
\begin{align}\label{220}
(-\mu_\Gamma^+,1,\dots,1)\in\partial\Gamma.
\end{align}
Note that $\mu_\Gamma^+$ is well-defined by \eqref{21'} and \eqref{22'}, is uniquely determined by $\Gamma$ and satisfies $\mu_\Gamma^+\in[0,n-1]$. We show that when $\mu_\Gamma^+>1$, the hyperbolic solution is indeed the unique solution. On the other hand, and somewhat more surprisingly, we show that when $\mu_\Gamma^+ \leq 1$, the solution set to \eqref{15} consists of a one parameter family depending only on $x_n$, of which the hyperbolic solution is the \textit{minimal} solution. We refer to Theorems \ref{B} and \ref{A} below for the precise statements. Note that for the G{\aa}rding cones $\Gamma_k^+$,  $\mu_{\Gamma_k^+}^+ = \frac{n-k}{k}$ and hence $\mu_{\Gamma_k^+}^+ \leq 1$ if and only if $k\geq \frac{n}{2}$. \medskip 

We now state our first main result, which demonstrates non-uniqueness of solutions in the case $\mu_\Gamma^+ \leq 1$ through the construction of a one-parameter family of solutions:

\begin{thm}\label{B}
	Let $(f,\Gamma)$ satisfy \eqref{21'}--\eqref{26'} and $\mu_\Gamma^+ \leq 1$. Then there exists a one-parameter family of smooth solutions $\{w^{(a)}(x_n)\}_{a\geq 0}$ to \eqref{15} satisfying the following six properties:\smallskip
		\begin{enumerate}
		\item $w^{(0)}(x_n) = x_n$,\smallskip
		\item $w^{(a)}(1) = 1+a$,\smallskip
		\item $(w^{(a)})' \geq 1$ (and hence $w^{(a)}(x_n) \geq x_n$) and $(w^{(a)})'(0)=1$,\smallskip
		\item $(w^{(a)})'' \geq 0$,\smallskip
		\item $0 \leq a<b \implies w^{(a)}(x_n) < w^{(b)}(x_n)$ for $x_n >0$,\smallskip
		\item $g_{w^{(a)}}$ is locally complete near $\partial\mathbb{R}_+^n$ for all $a \geq 0$ but incomplete on $\mathbb{R}_+^n$ for $a>0$. 
		\end{enumerate}
\end{thm}

In Theorem \ref{B}, local completeness near $\partial\mathbb{R}_+^n$ is understood in the following sense: for every point $x\in\partial\mathbb{R}_+^n$ and every neighbourhood $U$ of $x$ in $\overline{\mathbb{R}_+^n}$, any geodesic emanating from a point in $U$ which does not exit $U$ can be extended indefinitely. 

\begin{rmk}
	We will prove Theorem \ref{B} under a weaker assumption than \eqref{26'} -- see \eqref{23'''} in Section \ref{s3}. 
\end{rmk}

Our next main result demonstrates that when $\mu_\Gamma^+ > 1$, the hyperbolic solution is the \textit{unique} viscosity solution to \eqref{15}, and when $\mu_\Gamma^+\leq 1$, any viscosity solution to \eqref{15} depends only on $x_n$ and belongs to the family obtained in Theorem \ref{B}. Thus, we completely classify continuous viscosity solutions to \eqref{15}. We find this striking difference between the cases $\mu_\Gamma^+>1$ and $\mu_\Gamma^+ \leq 1$ surprising. Indeed, if one considers \eqref{15} for two pairs $(f_1,\Gamma_1)$ and $(f_2,\Gamma_2)$ with $\Gamma_2 \subset \Gamma_1$ (from which it is easy to see that $\mu_{\Gamma_2}^+ \leq \mu_{\Gamma_1}^+$), then the constraint $\lambda\in\Gamma_2$ is more restrictive than $\lambda\in\Gamma_1$, and thus one might expect that the equation corresponding to $(f_2,\Gamma_2)$ should admit fewer solutions. However, when $\mu_{\Gamma_2}^+ \leq 1 <  \mu_{\Gamma_1}^+$, our result shows that the equation corresponding to $(f_2,\Gamma_2)$ has infinitely many solutions, whereas the equation corresponding to $(f_1,\Gamma_1)$ has a unique solution. \medskip 

 In fact, our intermediate conclusion that a viscosity solution $w$ to \eqref{15} depends only on $x_n$ does not require the full strength of assumptions \eqref{21'}--\eqref{26'}. In particular, we do not require $\Gamma$ to be a convex cone, nor do we require concavity or homogeneity of $f$. We therefore introduce the following relaxed conditions to be assumed alongside \eqref{24'} and \eqref{25'}:
\begin{align}
& \Gamma\subset\mathbb{R}^n\text{ is an open, connected and symmetric set}, \label{21''} \tag{\ref{21'}$'$}\\
& \Gamma + \Gamma_n^+  \subseteq \Gamma,  \label{22''} \tag{\ref{22'}$'$} \\
& f\in C^\infty(\Gamma)\cap C^0(\overline{\Gamma}) \text{ is symmetric in the }\lambda_i. \label{23''} \tag{\ref{23'}$'$}
\end{align}
(In the case that $\Gamma = \mathbb{R}^n$, the condition that $f=0$ on $\partial\Gamma$ in \eqref{24'} is vacuous.)

\begin{thm}\label{A}
	Suppose that $(f,\Gamma)$ satisfies \eqref{21''}--\eqref{23''}, \eqref{24'} and \eqref{25'}, and suppose that $w\in C^0(\overline{\mathbb{R}_+^n})$ is a positive function satisfying \eqref{15} in the viscosity sense. Then $w$ depends only on $x_n$. Moreover, if $(f,\Gamma)$ also satisfies \eqref{21'}--\eqref{23'} and \eqref{23'''}, then $w$ is smooth and:\medskip
	\begin{enumerate}
		\item If $\mu_\Gamma^+ > 1$, then $w(x_n) = w^{(0)}(x_n) = x_n$ and thus $g_w$ is the hyperbolic metric on $\mathbb{R}_+^n$.\medskip
		\item If $\mu_\Gamma^+ \leq 1$, then $w=w^{(a)}$ for some $a \geq 0$, where $\{w^{(a)}\}_{a\geq 0}$ is the one-parameter family of solutions in Theorem \ref{B}. In particular, the hyperbolic solution $w^{(0)}(x_n) = x_n$ is the minimal solution to \eqref{15} and the only solution to \eqref{15} which yields a complete metric on $\mathbb{R}_+^n$.
	\end{enumerate}
\end{thm}

Our proof that $w=w(x_n)$ in Theorem \ref{A} uses a novel variant of the method of moving spheres, which in its original form has been applied to study Yamabe-type equations by Li \& Li \cite{LL05, LL06}, Li \& Zhang \cite{LZ03} and Li \& Zhu \cite{LZ95}, for example. (The method of moving spheres has its origins in the method of moving planes, which was developed in fundamental works such as \cite{Ser71, BN91, GNN79, GNN81, A58}.) However, attempting to replicate the methods used in \cite{LZ95, LZ03, LL05, LL06} presents significant challenges in our context. More precisely, in existing work for equations on the half-space, the moving spheres are centred on the hyperplane $\{x_n=0\}$. In attempting to implement this approach in our context (where the ellipticity degenerates at the boundary $\{x_n=0\}$), one needs to address three issues: \smallskip
\begin{enumerate}
	\item The existence of a small radius $r_0>0$ such that the inverted solution across a sphere $\partial B_{r_0}$ centred on the boundary $\{x_n=0\}$ lies below the solution in $\mathbb{R}_+^n\backslash B_{r_0}$ (so that one can start the method of moving spheres).\smallskip
	\item A Hopf lemma along the boundary $\{x_n=0\}$ (so as to ensure some separation between the solution and its inversion near $\{x_n=0\}$). \smallskip
	\item A suitable strong comparison principle which forbids `touching at infinity'. 
\end{enumerate} 
The Hopf lemma is false when $\mu_\Gamma^+ \leq 1$ as a consequence of Theorem \ref{B}: for $a<b$, we have $w^{(a)}<w^{(b)}$ in $\mathbb{R}_+^n$, but $w^{(a)} = w^{(b)}=0$ and $\partial_{x_n} w^{(a)}= \partial_{x_n}w^{(b)} =1$ on $\partial\mathbb{R}_+^n$. Moreover, it is not clear whether points 1 and 3 are true. A key novelty of our work is that we bypass these issues by using moving spheres centred in the \textit{lower} half-space, making use of the comparison principle established by Li, Nguyen \& Wang in \cite{LNW18}\footnote{Our method is not applicable in the settings considered by \cite{LL05, LL06, LZ03, LZ95}, since the corresponding comparison principle is not true.}. An immediate advantage of this method is that the domain in which we run our argument is bounded, and hence point 3 can be completely avoided. For the separation between the solution and its inversion, we do not draw on a Hopf lemma but instead prove suitable $C^0$ and $C^1$ bounds and asymptotics near $\partial\mathbb{R}_+^n$. We first establish the required $C^0$ bounds by constructing suitable comparison functions. To establish regularity and asymptotics near $\partial\mathbb{R}_+^n$, we combine our $C^0$ estimates with the small perturbation theorem of Savin \cite{Sav07}, following arguments similar to those of Li, Nguyen \& Xiong \cite{LNX22} -- see Section \ref{s2} for the details. Finally, it turns out that issues arising from point 1 are irrelevant for our method.\medskip 

Once we have shown $w=w(x_n)$, a delicate ODE analysis is then carried out in Section \ref{s3} to complete the proofs of Theorems \ref{B} and \ref{A}.  \medskip

As a consequence of \eqref{303} and Theorem \ref{A}, we have the following corollary which was previously observed in \cite[Remark 2.4]{HS20}:
\begin{cor}\label{AA}
	$u(x) = x_n^{-\frac{n-2}{2}}$ is the unique positive continuous viscosity solution to \eqref{303}.
\end{cor}

To place our results into context, let us now discuss some relevant literature on Liouville-type theorems. We first recall the classical result of Liouville \cite{Liou1853}, which asserts that a positive harmonic function on $\mathbb{R}^n$ is constant. In the seminal work of Caffarelli, Gidas \& Spruck \cite{CGS89}, it is shown that for $n\geq 3$, every positive solution to the Yamabe equation
\begin{align*}
-\Delta u = \frac{1}{4}n(n-2)u^{\frac{n+2}{n-2}} \quad \text{in }\mathbb{R}^n 
\end{align*}
is of the form 
\begin{align*}
u(x) = \bigg(\frac{a}{1+a^2|x-\overline{x}|^2}\bigg)^{\frac{n-2}{2}}
\end{align*}
for some $a>0$ and $\overline{x}\in\mathbb{R}^n$ (see also the earlier work of Gidas, Ni \& Nirenberg \cite{GNN79} and Obata \cite{Ob71} under additional assumptions on $u$ at infinity for $n\geq 3$, and the work of Chen \& Li \cite{CC91} for $n=2$). Generalisations to fully nonlinear equations of the form
\begin{align}\label{202}
f(\lambda(A_w)) = \frac{1}{2}, \quad \lambda(A_w)\in\Gamma
\end{align}
and
\begin{align}\label{204}
\lambda(A_w)\in\partial\Gamma
\end{align}
on $\mathbb{R}^n$ and $\mathbb{R}^n\backslash\{0\}$ have been considered in \cite{Via00b, CGY02b, CGY03, LL03, LL05, Li07, Li09, CLL23, FMW23}, for instance, and have played an important role in the existence theory for fully nonlinear Yamabe-type problems in works such as \cite{CGY02b, LL03, LL05, GW06, GV07, LN14}. In contrast to Theorem \ref{A}, when $(f,\Gamma)$ satisfies \eqref{21'}--\eqref{24'} the constant $\mu_\Gamma^+$ does not play a role in these Liouville-type results.\medskip

On the other hand, on $\mathbb{R}_+^n$ there are two geometrically motivated boundary conditions: either prescribing constant mean curvature on $\partial\mathbb{R}_+^n$ with respect to $g_w$ (which corresponds to a Neumann-type condition for $w$), or requiring that $g_w$ is locally complete near $\partial \mathbb{R}_+^n$ (which is usually formulated\footnote{In general, the completeness of $g_w$ and the vanishing of $w$ on the boundary are not equivalent.} by imposing $w=0$ on $\partial\mathbb{R}_+^n$). By the classical work of Schoen \& Yau \cite[Theorem 2.7]{SY88}, $\mathbb{R}_+^n$ admits no complete conformally flat metric with nonnegative scalar curvature. Therefore, there are no solutions to \eqref{202} or \eqref{204} on $\mathbb{R}_+^n$ which yield metrics that are locally complete near $\partial\mathbb{R}_+^n$. On the other hand, Liouville-type theorems on solutions to \eqref{202} and \eqref{204} on $\mathbb{R}_+^n$ with Neumann-type conditions have been considered in works such as \cite{LL06, Wei24, CLL24a, CLL24b,LZ95,CSF96,LZ03, CW18, E90}.\medskip 

 Theorem \ref{B}, Theorem \ref{A} and Corollary \ref{AA} fit into the realm of Liouville-type theorems for elliptic equations involving $f(\lambda(-A_w))$. In this direction, we first mention the seminal work of Osserman \cite{Oss57}, which asserts that there are no positive solutions to
\begin{align*}
\Delta u = \frac{1}{4}n(n-2)u^{\frac{n+2}{n-2}} \quad \text{on }\mathbb{R}^n.
\end{align*}
Jin, Li \& Xu \cite{JLX05} later extended this non-existence result equations of the form
\begin{align}\label{205}
f(\lambda(-A_w)) = \frac{1}{2}, \quad \lambda(-A_w)\in\Gamma 
\end{align}
on $\mathbb{R}^n$ for $(f,\Gamma)$ satisfying \eqref{21'}--\eqref{24'}. In recent work \cite{CLL23}, Chu, Li \& Li prove a Liouville-type theorem for the degenerate elliptic equation 
\begin{align*}
\lambda(-A_w)\in\partial\Gamma \quad \text{in }\mathbb{R}^n
\end{align*}
assuming $\mu_\Gamma^+\not=1$, and obtain counterexamples to the Liouville theorem whenever $\mu_\Gamma^+=1$. In fact, the authors obtain such results in a more general context, and we refer the reader to \cite{CLL23} for the details.\medskip

Returning once again to the setting of $\mathbb{R}_+^n$, we observe that the situation for the two previously mentioned boundary conditions is somewhat reversed when compared to the `positive cases' \eqref{202} and \eqref{204}. As far as we are aware, Theorem \ref{A} is the first Liouville-type result for the fully nonlinear equation \eqref{205} on $\mathbb{R}_+^n$ with the homogeneous Dirichlet boundary condition. In contrast, the non-existence of solutions to \eqref{205} on $\mathbb{R}_+^n$ under Neumann-type conditions was established in the semilinear case by Lou \& Zhu \cite{LouZhu99} and in the fully nonlinear case by Jin, Li \& Xu \cite{JLX05}. For further Liouville-type results in the semilinear case on $\mathbb{R}_+^n$, see e.g.~Lou \& Zhu \cite{LouZhu99} and Tang, Wang \& Zhu \cite{TWZ22}.\medskip

Part of our motivation for establishing Theorems \ref{B} and \ref{A} is in relation to a fully nonlinear version of the Loewner-Nirenberg problem. Recall that the Schouten tensor of a Riemannian metric $g$ is defined by
\begin{align*}
A_g = \frac{1}{n-2}\bigg(\operatorname{Ric}_g - \frac{R_g}{2(n-1)}g\bigg),
\end{align*} 
where $\operatorname{Ric}_g$ and $R_g$ denote the Ricci curvature and scalar curvature of $g$, respectively. Note that if $g_w = w^{-2}|dx|^2$, then $g_w^{-1}A_{g_w}$ coincides with the quantity $A_w$ defined in \eqref{305}.\medskip

\noindent\textbf{The fully nonlinear Loewner-Nirenberg problem:} \textit{Given a smooth compact Riemannian manifold $(M,g_0)$ of dimension $n\geq 3$ with smooth non-empty boundary, and given $(f,\Gamma)$ satisfying \eqref{21'}--\eqref{26'}, does there exist a conformal metric $g = w^{-2}g_0$ which is complete in the interior of $M$ and satisfies}
\begin{align}\label{105}
f(\lambda(-g^{-1}A_g)) = \frac{1}{2}, \quad \lambda(-g^{-1}A_g)\in\Gamma \quad \text{on }M 
\end{align}
\textit{and}
\begin{align}
\lim_{\mathrm{d}_{g_0}(x,\partial M)\rightarrow 0} \frac{w(x)}{\mathrm{d}_{g_0}(x,\partial M)} \in(0,\infty).
\end{align}

In their classical work \cite{LN74}, Loewner \& Nirenberg established (among other results) that a smooth bounded domain in $\mathbb{R}^n$ ($n\geq 3$) admits a smooth complete conformally flat metric with constant negative scalar curvature (the case $(f,\Gamma) = (\frac{1}{n}\sigma_1,\Gamma_1^+)$). Aviles \& McOwen \cite{AM88} later extended this result to compact Riemannian manifolds with boundary. For some further references in the scalar curvature case, see the introduction in \cite{DN23}. For $(f,\Gamma)$ satisfying \eqref{21'}--\eqref{26'}, the existence of a solution to the fully nonlinear Loewner-Nirenberg problem is known e.g.~in the following situations:\smallskip
\begin{itemize}
	\item A particular case in which $\mu_\Gamma^+ \geq 1$ and \eqref{105} can be recast as an equation involving $\operatorname{Ric}_g$ \cite{GSW11},\smallskip
	\item When $M$ is a bounded subset of $\mathbb{R}^n$ with smooth boundary \cite{GLN18},\smallskip
	\item When $\mu_\Gamma^+>1$ \cite{DN23}.
\end{itemize}
See also the combination of results in \cite{Yuan22, Yuan24}, which yield existence when $\mu_\Gamma^+ \geq 1$, $(1,0,\dots,0)$ $ \in\Gamma$ and $\Gamma$ satisfies a further technical assumption (the latter condition is only used to obtain asymptotics of solutions near $\partial M$). \medskip 

In \cite{DN25b}, we show that the fully nonlinear Loewner-Nirenberg problem admits a solution whenever $\mu_\Gamma^+>1-\delta$, where $\delta$ is a constant depending on certain global geometric data. \medskip

We now return to the context of Theorems \ref{B} and \ref{A}. A large part of our existence proof in \cite{DN23} when $\mu_\Gamma^+> 1$ concerns a local boundary $C^0$ estimate for solutions to \eqref{105} with positive Dirichlet boundary data for the conformal factor $w$ -- see the proof of Proposition 3.3 in \cite{DN23}, which uses local barriers constructed in Proposition 3.4 therein. In looking to extend our existence results to the case $\mu_\Gamma^+ \leq 1$, we were therefore led to consider the possibility of obtaining such boundary estimates when $\mu_\Gamma^+ \leq 1$. Our final main result demonstrates that the one parameter family of solutions obtained in Theorem \ref{B} serve as \textit{counterexamples} not only to the local boundary $C^0$ estimates, but also to local boundary gradient estimates when $\mu_\Gamma^+ \leq 1$:

\begin{thm}\label{D}
	Suppose that $(f,\Gamma)$ satisfies \eqref{21'}--\eqref{26'} and $\mu_\Gamma^+ \leq 1$. Let $\{w^{(a)}\}_{a\geq 0}$ denote the one parameter family of solutions to \eqref{15} from Theorem \ref{B}. Then for all $0<\ep<\mathcal{E}$,
	\begin{align*}
	\inf_{t\in[\ep,\mathcal{E}]} w^{(a)}(t) \rightarrow+\infty \quad \text{and} \quad \inf_{t\in[\ep,\mathcal{E}]} (w^{(a)})'(t) \rightarrow+\infty \quad \text{as }a\rightarrow\infty.
	\end{align*}
\end{thm}

	The plan of the paper is as follows. In Section \ref{s2}, we prove $C^0$ bounds, regularity and symmetry of solutions to the fully nonlinear Loewner-Nirenberg problem \eqref{15} in the upper half-space. In particular, we prove Proposition \ref{A'}, which constitutes a large portion of the work towards the proofs of Theorems \ref{B}, \ref{A} and \ref{D}. Having reduced the fully nonlinear Loewner-Nirenberg problem in the upper half-space to a one-dimensional problem in Section \ref{s2}, we then carry out a delicate ODE analysis to prove Theorems \ref{B}, \ref{A} and \ref{D} in Section \ref{s3}. We also provide in Appendices \ref{appa} and \ref{appb} some further results which may be of independent interest.\medskip 
	
	\noindent\textbf{Acknowledgements:} The authors would like to thank YanYan Li and Piotr T.~Chru\'sciel for helpful discussions. \medskip 
	
	\noindent\textbf{Rights retention statement:} For the purpose of Open Access, the authors have applied a CC BY public copyright licence to any Author Accepted Manuscript (AAM) version arising from this submission.

\section{$C^0$ bounds, regularity and symmetry}\label{s2}

The goal of this section is to prove the following result, which forms a significant step towards the proof of Theorems \ref{B}, \ref{A} and \ref{D}:

\begin{prop}\label{A'}
	Suppose that $(f,\Gamma)$ satisfies \eqref{21''}--\eqref{23''}, \eqref{24'} and \eqref{25'}, and suppose that $w\in C^0(\overline{\mathbb{R}_+^n})$ satisfies \eqref{15} in the viscosity sense. Then $w = w(x_n)$ and there exist constants $\ep>0$ and $C = C(\ep) \geq 0$ such that:\smallskip
	\begin{enumerate}
		\item $w(x_n) \geq x_n$, \smallskip
		\item $w(x_n) \leq x_n + Cx_n^2$ in $\{0\leq x_n \leq \ep\}$,\smallskip
		\item $w\in C^{2,\alpha}_{\operatorname{loc}}(\{0<x_n \leq \ep\}) \cap C^1(\{0 \leq x_n \leq \ep\})$. \smallskip
	\end{enumerate} 
In particular, $w'(0) = 1$ and $g_w$ is locally complete near $\partial\mathbb{R}_+^n$. 
\end{prop}

The proof of Proposition \ref{A'} proceeds according to the following steps. First, in Section \ref{cp}, we give a version of the comparison principle that will be used at various stages in our arguments. In Section \ref{bar}, we give a barrier construction that will be used in Section \ref{c0}. In Section \ref{c0}, we show that $w$ satisfies the lower bound $w(x) \geq x_n$ in $\mathbb{R}_+^n$ (Lemma \ref{l1}) and the upper bound $w(x) \leq x_n + Cx_n^2$ near $\partial\mathbb{R}^n_+$ (Lemma \ref{l2}). In Section \ref{reg}, we argue as in the work of Li, Nguyen \& Xiong \cite{LNX22} by using the small perturbation theorem of Savin \cite{Sav07} to show that $w$ is $C^{2,\alpha}_{\operatorname{loc}}$ near $\partial\mathbb{R}^n_+$ and $C^1$ near $\partial\mathbb{R}^n_+$ (Lemma \ref{17}). Finally, in Section \ref{sym} we apply the method of moving spheres to show that $w=w(x_n)$ (Lemma \ref{19}). As explained in the introduction, a key difference between our application of the method of moving spheres and that in related works such as Li \& Zhang \cite{LZ03} and Li \& Zhu \cite{LZ95} is that our moving spheres are centred in the \textit{lower} half-space rather than the hyperplane $\{x_n=0\}$. Carrying out this procedure relies on the bounds obtained near $\partial\mathbb{R}^n_+$ in Section \ref{c0} and the regularity obtained near $\partial\mathbb{R}^n_+$ in Section \ref{reg}. 

\subsection{Comparison principles}\label{cp}

We first recall the notion of a viscosity (sub/ super-)solution, introduced in the context of fully nonlinear Yamabe-type equations by Li in \cite{Li09}:
\begin{defn}\label{137}
	We say that $w_1 \in C^0(\Omega)$ is a viscosity supersolution to the equation 
	\begin{align}\label{506}
	f(\lambda(-A_w)) = \frac{1}{2}, \quad \lambda(-A_w)\in\Gamma \quad \text{on }\Omega
	\end{align}
	if for any $x_0\in\Omega$ and $\phi\in C^2(\Omega)$ satisfying $\phi(x_0) = w_1(x_0)$ and $\phi \leq w_1$ near $x_0$, it holds that 
	\begin{align*}
	f(\lambda(-A_\phi))(x_0) \geq \frac{1}{2} \quad \text{and} \quad  \lambda(-A_\phi)(x_0)\in\Gamma.  
	\end{align*} 
	We say that $w_2\in C^0(\Omega)$ is a viscosity subsolution to \eqref{506} if for any $x_0\in\Omega$ and $\phi\in C^2(\Omega)$ satisfying $\phi(x_0) = w_2(x_0)$ and $\phi \geq w_2$ near $x_0$, it holds that 
	\begin{align*}
	f(\lambda(-A_\phi))(x_0) \leq \frac{1}{2} \quad\text{and}\quad \lambda(-A_\phi)(x_0)\in\Gamma  
	\end{align*}
	or
	\begin{align*}
	\lambda(-A_\phi)(x_0)\in\mathbb{R}^n\backslash\Gamma. 
	\end{align*}
	We call $w$ a viscosity solution to \eqref{506} if it is both a viscosity supersolution and a viscosity subsolution. 
\end{defn}

We will often refer to the following comparison principle:
\begin{prop}\label{800}
	Suppose that $(f,\Gamma)$ satisfies \eqref{21''}--\eqref{23''} and \eqref{24'}, and let $\Omega$ be a bounded domain. Suppose $w_1 \in C^0(\overline{\Omega})$ is a viscosity supersolution to \eqref{506}, $w_2 \in C^0(\overline{\Omega})$ is a viscosity subsolution to \eqref{506}, $w_1,w_2>0$ in $\Omega$ and $w_2 < w_1$ on $\partial\Omega$. Then $w_2 < w_1$ in $\overline{\Omega}$. 
\end{prop}

The proof of Proposition \ref{800} uses another version of the comparison principle, which follows from the principle of propagation of touching points in \cite[Theorem 3.2]{LNW18}, as explained in \cite[Proposition 2.2]{GLN18}: 

\begin{prop}\label{510}
	Suppose that $(f,\Gamma)$ satisfies \eqref{21''}--\eqref{23''} and \eqref{24'}, and let $\Omega$ be a bounded domain. Suppose $w_1 \in C^0(\overline{\Omega})$ is a viscosity supersolution to \eqref{506}, $w_2 \in C^0(\overline{\Omega})$ is a viscosity subsolution to \eqref{506}, $w_1,w_2>0$ in $\overline{\Omega}$ and $w_2 \leq w_1$ on $\partial\Omega$. Then $w_2 \leq w_1$ in $\overline{\Omega}$. 
\end{prop}

\begin{proof}[Proof of Proposition \ref{800}]
	For $\ep>0$, let $\Omega_\ep = \{x\in\Omega:\operatorname{dist}(x,\partial\Omega)>\ep\}$. Then $0<w_2 < w_1$ on $\partial\Omega_\ep$ for $\ep$ sufficiently small, and thus Proposition \ref{510} implies $w_2 \leq w_1$ in $\Omega_\ep$. Taking $\ep\rightarrow 0$ we therefore see that $w_2 \leq w_1$ in $\Omega$. But since $w_2<w_1$ on $\partial\Omega$, the principle of propagation of touching points \cite[Theorem 3.2]{LNW18} then implies $w_2<w_1$ in $\Omega$. 
\end{proof}

\subsection{Barrier functions}\label{bar}

Suitable barrier functions will play an important role in obtaining $C^0$ estimates in the next section. For the global lower bound, we will simply use the hyperbolic metric on a sequence of larger and larger balls as lower barriers. For the upper bound near $\partial\mathbb{R}_+^n$, we will use supersolutions on annuli centred in the lower half space as upper barriers (cf.~a related but different construction in \cite[Proposition 3.4]{DN23}, and see also Appendix \ref{appb} for non-existence of related supersolutions when $\mu_\Gamma^+ \leq 1$). These supersolutions are constructed in the following lemma:

\begin{lem}\label{l3}
	Fix $x_0'\in\{x_n=0\}$ and suppose $R>200$, $\frac{r_1^2}{R^2} < \frac{1}{200}$ and $ \sqrt{R^2 - r_1^2} \leq r < R$. Let $A_{r, \sqrt{R^2 + r_1^2}}$ denote the open annulus centred at $(x_0', -R)$ with inner radius $r$ and outer radius $\sqrt{R^2 + r_1^2}$, and for $x = (x',x_n)$ let $d(x) = \sqrt{|x' - x_0'|^2 + (x_n + R)^2}$ denote the distance to its centre. Then for any constant $C$ satisfying
	\begin{align}\label{11}
	\frac{9R^2}{r_1^4} < C < \frac{R^3}{21r_1^4}, 
	\end{align}
	the conformal metric $g_{h_{r}}$ defined on $A_{r, \sqrt{R^2 + r_1^2}}$ by
	\begin{align*}
	h_{r}(x) = h_{r}(d) = (d-r) + C(d-r)^2 > 0 
	\end{align*}
	satisfies
	\begin{align*}
	\begin{cases}
	f(\lambda(-A_{h_{r}})) \geq \frac{1}{2}, \quad \lambda(-A_{h_r})\in\Gamma & \text{in }A_{r, \sqrt{R^2 + r_1^2}} \\
	h_{r} = 0 & \text{on }\{d=r\} \\
	h_{r} \geq 1 & \text{on }\{d=\sqrt{R^2 + r_1^2}\}. 
	\end{cases}
	\end{align*}
\end{lem}

\begin{proof}
	The fact that $h_r = 0$ on $\{d=r\}$ is clear from the definition of $h_r$. To see that $h_r \geq 1$ on $\{d=\sqrt{R^2+r_1^2}\}$, it suffices to show that $C(\sqrt{R^2 + r_1^2} - R)^2 \geq 1$ under our hypotheses on $R, r_1, r$ and $C$. To this end, note that since $\sqrt{1+x} \geq 1 + \frac{x}{3}$ for $x\in[0,3]$ and since we assume $\frac{r_1^2}{R^2}< \frac{1}{200}$, we have the inequality
	\begin{align}\label{12}
	\sqrt{R^2 + r_1^2} = R\bigg(1+ \frac{r_1^2}{R^2}\bigg)^{1/2} \geq R\bigg(1+\frac{r_1^2}{3R^2}\bigg).
	\end{align}
	The inequality in \eqref{12} combined with the lower bound on $C$ in \eqref{11} then clearly implies $C(\sqrt{R^2 + r_1^2} - R)^2 \geq 1$, as required.\medskip

	To complete the proof of the claim, it remains to show that our assumptions on $R, r_1, r$ and $C$ imply $f(\lambda(-A_{h_{r}})) \geq \frac{1}{2}$ and $\lambda(-A_{h_{r}})\in\Gamma$ in $A_{r, \sqrt{R^2 + r_1^2}}$. We start by computing the eigenvalues of $-A_{h_{r}}$, which have multiplicity 1 and $(n-1)$, respectively: 
	\begin{align*}
	\lambda_1 & = -h_{r}h_{r}'' + \frac{1}{2}(h_{r}')^2 = -2C\big((d-r) + C(d-r)^2\big) + \frac{1}{2}\big(1+2C(d-r)\big)^2 = \frac{1}{2}
	\end{align*}
	and
	\begin{align*}
	\lambda_2 & = -\frac{h_{r}h_{r}'}{d} + \frac{1}{2}(h_{r}')^2   = -\big((d-r) + C(d-r)^2\big)\frac{1+2C(d-r)}{r}+ \frac{1}{2}\big(1+2C(d-r)\big)^2 \nonumber \\
	& = \frac{1}{2} + (d-r)\bigg[2C -\frac{1}{d} + \Big(\frac{-3C}{d} + 2C^2\Big)(d-r) - \frac{2C^2}{d}(d-r)^2\bigg] \nonumber \\
	& \geq \frac{1}{2} + (d-r)\bigg[2C -\frac{1}{d} -\frac{3C}{d}(d-r) - \frac{2C^2}{d}(d-r)^2\bigg] \nonumber \\
	& \geq \frac{1}{2} + (d-r)\bigg[2C -\frac{7}{4d} - \frac{5C^2}{d}(d-r)^2\bigg],
	\end{align*}
	where we have applied the Cauchy-Schwarz inequality to reach the last line. It therefore suffices to show that
	\begin{align}\label{500}
	\phi \defeq 2C -\frac{7}{4d} - \frac{5C^2}{d}(d-r)^2  \geq 0. 
	\end{align} 
	To prove \eqref{500} we first note
	\begin{align*}
	R\bigg(1-\frac{r_1^2}{R^2}\bigg) \leq \sqrt{R^2 - r_1^2}\, \leq r < d < \sqrt{R^2 + r_1^2} \, \leq R\bigg(1+\frac{r_1^2}{R^2}\bigg),
	\end{align*} 
	which implies $d-r \leq \frac{2r_1^2}{R}$. These, together with $\frac{r_1^2}{R^2} < \frac{1}{200}$, imply
	\begin{align*}
	\frac{7}{4d} < \frac{7}{4R(1-\frac{r_1^2}{R^2})} < \frac{2}{R} \quad \text{and}\quad 	\frac{5C^2}{d}(d-r)^2 < \frac{5C^2}{R(1-\frac{r_1^2}{R^2})}\frac{4r_1^4}{R^2} < \frac{21C^2 r_1^4}{R^3}. 
	\end{align*}
	Therefore 
	\begin{align*}
	\phi \geq 2C - \frac{2}{R} - \frac{21C^2 r_1^4}{R^3},
	\end{align*}
	and thus $\phi$ is clearly seen to be nonnegative by \eqref{11}. 
\end{proof}

\subsection{$C^0$ bounds near $\partial\mathbb{R}^n_+$}\label{c0}

In this subsection we show that any continuous viscosity solution to \eqref{15} satisfies $w(x_n) \geq x_n$ on $\mathbb{R}_+^n$ and $w(x) \leq x_n + Cx_n^2$ near $\partial\mathbb{R}_+^n$, where $x = (x',x_n)\in\mathbb{R}^{n-1}\times[0,\infty) = \mathbb{R}_+^n$. We begin with the lower bound:

\begin{lem}[Global lower bound]\label{l1}
	Suppose that $(f,\Gamma)$ satisfies \eqref{21''}--\eqref{23''}, \eqref{24'} and \eqref{25'}, and that $w\in C^0(\overline{\mathbb{R}_+^n})$ satisfies \eqref{15} in the viscosity sense. Then $w(x',x_n) \geq x_n$ in $\mathbb{R}_+^n$. 
\end{lem}
\begin{proof}
	It suffices to show that for any $R>0$ and $x_n<2R$,
	\begin{align*}
	w(x',x_n) \geq x_n - \frac{x_n^2}{2R}. 
	\end{align*}
	Once this inequality is obtained, the conclusion follows from taking $R\rightarrow+\infty$. \medskip

	 As mentioned in the previous section, we use the hyperbolic metric on a sequence of larger and larger balls as lower barriers. To this end, fix $(x_0', R)\in\mathbb{R}_+^n$, and for $r < R$ let $B_r$ denote the ball centred at $(x_0', R)$ of radius $r$. For $x = (x', x_n)$, denote by $d(x) = \sqrt{|x'-x_0'|^2 + (x_n - R)^2}$ the distance to the centre of $B_r$, and consider on $B_r$ the conformal metric $g_{h_r}$ defined by
	\begin{align*}
	h_r (x) = h_r (d) = \frac{r^2- d^2}{2r}. 
	\end{align*}
	Since  $g_{h_r}$ is the hyperbolic metric on $B_r$, we have $f(\lambda(-A_{h_r}))  = \frac{1}{2}$ in $B_r$. (One can show this by direct computation: using primes to denote derivatives with respect to $d$, it is easy to see that the eigenvalues of $-A_{h_r}$ are $\lambda_1 = -h_r h_r'' + \frac{(h_r')^2}{2} = \frac{1}{2}$ (with multiplicity one) and $\lambda_2 = -\frac{h_r h_r'}{d} + \frac{(h_r')^2}{2} = \frac{1}{2}$ (with multiplicity $n-1$).) Clearly $h_r = 0$ on $\partial B_r$, and since $r<R$ we have $w>0$ on $\partial B_r$. Therefore we may apply the comparison principle in Proposition \ref{800} to assert $w>h_r$ in $B_r$, and taking $r\rightarrow R$ we see that $w \geq h_R$ in $B_R$. In particular, for $x_n < 2R$ we have
	\begin{align*}
	w(x_0', x_n) \geq h_R(x_0', x_n) = \frac{R^2 - (x_n - R)^2}{2R} = x_n - \frac{x_n^2}{2R},
	\end{align*}
	as required. 
\end{proof}

Next we prove an upper bound in a half-ball centred at a given point in $\partial\mathbb{R}_+^n$: 

\begin{lem}[Upper bound near boundary]\label{l2}
	Suppose that $(f,\Gamma)$ satisfies \eqref{21''}--\eqref{23''}, \eqref{24'} and \eqref{25'}, and that $w\in C^0(\overline{\mathbb{R}_+^n})$ satisfies \eqref{15} in the viscosity sense. Then for all $x_0' \in \{x_n=0\}$, there exist constants $C_0>1$ and $r_0<1$ (depending on $x_0'$ and the modulus of continuity of $w$ near $x_0'$) such that $w(x', x_n) \leq x_n + C_0x_n^2$ in the half-ball $\{|x'-x'_0|^2 + x_n^2 < r_0^2,\, x_n\geq 0\}$. 
\end{lem}

\begin{proof}
	As mentioned in the previous section, we use the supersolutions constructed in Lemma \ref{l3} as upper barriers near $\partial\mathbb{R}_+^n$. To this end, fix $x_0'\in\{x_n=0\}$. Then there exists $r_1>0$ (depending on the modulus of continuity of $w$ near $x_0'$) such that $w<1$ in the half-ball $\{|x'-x_0'|^2 + x_n^2 \leq r_1^2,\, x_n \geq 0\}$. Fix $R>200$ satisfying $\frac{r_1^2}{R^2}< \frac{1}{200}$, and for $\sqrt{R^2 - r_1^2} \leq r < R$ let $h_r$ denote the function constructed in Lemma \ref{l3} with these values of $r_1$ and $R$ for some fixed constant $C$ satisfying \eqref{11}. Then by applying the comparison principle in Proposition \ref{800} in the spherical cap $A_{r, \sqrt{R^2 + r_1^2}} \cap \{x_n > 0\}  = \{d<\sqrt{R^2 + r_1^2},\, x_n>0\}$ (which is clearly independent of $r$, since $r<R$), we obtain $w \leq h_{r}$ here. Taking $r \rightarrow R$, we therefore obtain in $\{d<\sqrt{R^2 + r_1^2}, x_n>0\}$ the upper bound
	\begin{align*}
	w \leq h_R =  (d-R) + C(d-R)^2. 
	\end{align*}
	In particular, for our fixed value of $x_0'$ we have the bound $w(x_0', x_n) \leq x_n + Cx_n^2$ in the region $\{d<\sqrt{R^2 + r_1^2},\, x_n>0\}$. By the upper bound for $C$ in \eqref{11}, the conclusion of the lemma is then readily seen.
\end{proof}

\subsection{Regularity near $\partial\mathbb{R}_+^n$}\label{reg}

We now use the $C^0$ bounds established in the previous section to show:

\begin{lem}\label{17}
	Suppose that $(f,\Gamma)$ satisfies \eqref{21''}--\eqref{23''}, \eqref{24'} and \eqref{25'}, and that $w\in C^0(\overline{\mathbb{R}_+^n})$ satisfies \eqref{15} in the viscosity sense. Fix $x_0' \in \{x_n=0\}$ and $R>0$. Then:\smallskip
	\begin{enumerate}
		\item There exists a constant $r_0>0$ (depending on $x_0'$, $R$ and the modulus of continuity of $w$ near $x_0'$) such that $w\in C_{\operatorname{loc}}^{2,\alpha}(\{|x'-x_0'| < R,\,0<x_n < r_0\})$,\smallskip
		\item And
		\begin{align*}
		\sup_{\{|x'-x_0'| < R,\, 0<x_n < r\}}|\nabla w(x',x_n) - (0,\dots,0,1)| \rightarrow 0 \quad \text{as }r\rightarrow 0. 
		\end{align*}
	\end{enumerate}
\end{lem}
\begin{proof}
	Having established the $C^0$ bounds in Lemmas \ref{l1} and \ref{l2}, the proof of Lemma \ref{17} largely follows the argument set out in \cite{LNX22}. We provide the details here for completeness, but the reader who is familiar with \cite{LNX22} may safely skip this subsection. \medskip 
	
	\noindent\textbf{Step 1:} In this step we prove the first statement in the proposition. Without loss of generality we take $x_0'$ to be the origin. Fix $0<\delta< \frac{\sqrt{2}}{2}r_0$ (where $r_0$ is as in Lemma \ref{l2}) and define the wedges
	\begin{align*}
	\Omega_{\delta} = \{x=(x',x_n)\in\mathbb{R}^n_+: |x'| < x_n < \delta\}
	\end{align*}
	and
	\begin{align*}
	\widetilde{\Omega}_{\delta} = \{x=(x',x_n)\in\mathbb{R}^n_+: 2|x'| < x_n < \delta\} \subset \Omega_\delta.
	\end{align*}
	Note that the assumption $\delta < \frac{\sqrt{2}}{2}r_0$ ensures (by Lemmas \ref{l1} and \ref{l2}) that $x_n \leq w(x) \leq x_n + Cx_n^2$ in $\Omega_{\delta}$. In keeping with \cite{LNX22}, we write $e^{2u} = w^{-2}$, so that this bound on $w$ is equivalent to $-\ln(1+Cx_n) \leq u(x) + \ln x_n \leq 0$, which implies 
	\begin{align}\label{21}
	|u(x) + \ln x_n| \leq Cx_n \quad \text{in }\Omega_\delta.
	\end{align}

	For a given function $v$, denote
	\begin{align*}
	A^v = -\nabla^2 v+ dv\otimes dv - \frac{1}{2}|\nabla v|^2 I
	\end{align*}
	and
	\begin{align*}
	\hat{A}^v(x) & \defeq x_n^2 A^v(x) -x_n(e_n\otimes dv + dv\otimes e_n) + x_n e_n \partial_n v\, I - \frac{1}{2}I \nonumber \\
	& = x_n^2\Big(\!-\!\nabla^2 v + dv\otimes dv - \frac{1}{2}|\nabla v|^2 I\Big)-x_n(e_n\otimes dv + dv\otimes e_n) + x_n e_n \partial_n v\, I - \frac{1}{2}I. 
	\end{align*}
	Then if $\hat{u}(x) \defeq u(x) + \ln x_n$, it is easy to see that
	\begin{align*}
	\hat{A}^{\hat{u}}(x) = x_n^2 A^u (x).
	\end{align*}
	Therefore, if we define
	\begin{align*}
	&G(M,p,z,x) \nonumber \\
	 &\defeq f\bigg[\!-\!\lambda\bigg(x_n^2\Big(\!-\!M + p\otimes p - \frac{1}{2}|p|^2 I\Big)-x_n(e_n\otimes p + p\otimes e_n)   + x_n e_n p_n\, I - \frac{1}{2}I\bigg)\bigg] - \frac{1}{2}e^{2z},
	\end{align*}
	we have $G(0,0,0,x) = 0$ and 
	\begin{align*}
	G(\nabla^2 \hat{u}, \nabla \hat{u}, \hat{u},x) = f(\lambda(-\hat{A}^{\hat{u}})) - \frac{1}{2}e^{2\hat{u}} = x_n^2\bigg(f(\lambda(-A^u)) - \frac{1}{2}e^{2u}\bigg) = 0. 
	\end{align*}
	
	Given $r\geq 1$, let us now consider on $\Omega_{r\delta}$ the rescaled function 
	\begin{align*}
	\hat{u}_r(x) \defeq \hat{u}\bigg(\frac{x}{r}\bigg) = u\bigg(\frac{x}{r}\bigg) + \ln\bigg(\frac{x_n}{r}\bigg). 
	\end{align*}
	Then $\hat{A}^{\hat{u}_r}(x) = \hat{A}^{\hat{u}}(\frac{x}{r})$ and therefore
	\begin{align}\label{20}
	G(\nabla^2 \hat{u}_r(x), \nabla \hat{u}_r(x), \hat{u}_r(x), x) = G\bigg(\nabla^2 \hat{u}\Big(\frac{x}{r}\Big), \nabla \hat{u}\Big(\frac{x}{r}\Big), \hat{u}\Big(\frac{x}{r}\Big), \frac{x}{r}\bigg) = 0
	\end{align}
	in $\Omega_{r\delta}$. In particular, \eqref{20} holds in $\Sigma_{r\delta}\defeq\{x\in \Omega_{r\delta}: |x_n-1|<\frac{1}{2}\}$, which is non-empty for e.g.~$r>2\delta^{-1}$. Moreover, we have the following: there exists a small constant $\eta_0>0$ and a large constant $C_0>1$ (depending only on $n$ and $f$) such that if $\|M\| + |p| + |z| < \eta_0$ and $x\in \Sigma_{r\delta}$ for some $r> 2\delta^{-1}$, then 
	\begin{align*}
	\frac{1}{C_0}I \leq \bigg(\frac{\partial G}{\partial M_{ij}}(M,p,z,x)\bigg)_{ij} \leq C_0 I,
	\end{align*}
	or in other words $G$ is uniformly elliptic in this range. This follows from the fact that $G$ is uniformly elliptic at $(0,0,0,x)$ for $x$ bounded away from 0 and infinity. In addition, \eqref{21} tells us $|\hat{u}(x)| \leq Cx_n$ in $\Omega_{\delta}$, thus $|\hat{u}_r(x)| \leq C\frac{x_n}{r}$ in $\Omega_{r\delta}$, and therefore
	\begin{align*}
	|\hat{u}_r(x)| \leq \frac{C_1}{r} \quad \text{in }\Sigma_{r\delta}
	\end{align*}
	for a constant $C_1$ independent of $r$. We may then apply the small perturbation result of Savin \cite{Sav07}: for all $\eta_1<\eta_0$, there exists $r_1 = r_1(f,n,\eta_1,C_0, C_1)$ large such that 
	\begin{align*}
	\|\hat{u}_r\|_{C^{2,\alpha}(\widetilde{\Sigma}_{r\delta})} \leq \eta_1 \quad \text{for all }r \geq r_1,
	\end{align*}
	where $\widetilde{\Sigma}_{r\delta} = \{x\in \widetilde{\Omega}_{r\delta}: |x_n-1|<\frac{1}{4}\}$. Therefore, if $\ep = \frac{\delta}{r_1}$ then $u\in C^{2,\alpha}_{\operatorname{loc}}(\widetilde{\Omega}_{\ep})$. The first statement of the proposition therefore holds by a standard covering argument. \medskip

	\noindent\textbf{Step 2:} In this step we prove the second statement in the proposition. Fix a ball $B(y, r_1) \subset \{|x'-x_0'| < R,\,0<x_n < \widetilde{r}\}$, where $\widetilde{r}$ is sufficiently small such that 
	\begin{align}\label{22}
	|u(x) + \ln x_n| \leq Cx_n \quad \text{in }  \{|x'-x_0'| < R,\,0<x_n < \widetilde{r}\}
	\end{align}
	(note that $\widetilde{r}$ exists by Lemma \ref{l2}).\medskip 
	
	Now define $h(x) = -\ln x_n$, and for $x\in B(0,1)$ define 
	\begin{align*}
	u_y(x) = \ln r_1 + u(y+r_1 x) \quad \text{and} \quad h_y(x) & = \ln r_1 + h(y+r_1 x) \nonumber \\
	& = \ln r_1 - \ln (y_n + r_1 x_n). 
	\end{align*}
	Then if we define $F$ by
	\begin{align*}
	F(\nabla^2 v, \nabla v, v, x) = f(\lambda(-A^v)) - \frac{1}{2}e^{2v},
	\end{align*}
	we have $F(\nabla^2 u, \nabla u , u, x) = F(\nabla^2 h, \nabla h, h, x) = 0$ and hence
	\begin{align*}
	F(\nabla^2 u_y, \nabla u_y, u_y, x) = F(\nabla^2 h_y, \nabla h_y, h_y, x) = 0 \quad \text{in }B(0,1).
	\end{align*}
It follows that
	\begin{align*}
	L_y (u_y - h_y) = 0 \quad\text{in }B(0,1)
	\end{align*}
	where $L_y = \sum_{i,j} a^{ij}_y(x)\partial_i\partial_j + \sum_i b^i_y(x)\partial_i + c_y(x)$ and
	\begin{align*}
	a^{ij}_y(x) & = \int_0^1 \frac{\partial}{\partial M_{ij}}F(\nabla^2 \xi_{t,y}, \nabla \xi_{t,y}, \xi_{t,y}, x)\,dt, \nonumber \\
	b^i_y(x) & = \int_0^1 \frac{\partial}{\partial p_i}F(\nabla^2 \xi_{t,y}, \nabla \xi_{t,y}, \xi_{t,y}, x)\,dt, \nonumber \\
	c_y(x) & = \int_0^1 \frac{\partial}{\partial z}F(\nabla^2 \xi_{t,y}, \nabla \xi_{t,y}, \xi_{t,y}, x)\,dt, \nonumber \\
	\xi_{t,y} & = tu_y + (1-t)h_y.
	\end{align*}
	Now, we know by \eqref{22} that 
	\begin{align*}
	|u_y(x) - h_y(x)| = |u(y+r_1x) + \ln (y+r_1x)_n| \leq C(y+r_1 x)_n \leq Cr_1 
	\end{align*}
	for $x\in B(0,1)$, where $C$ is independent of $y$ (since we assume $0< y_n < \widetilde{r}$) and $r_1$. Therefore, Schauder theory implies 
	\begin{align*}
	|\nabla^j(u_y - h_y)(x)| \leq C_jr_1 \quad \text{for }j\geq 1, \, x\in B(0, 1/2)
	\end{align*}
	and so after rescaling
	\begin{align*}
	|\nabla^j(u-h)(x)| = |\nabla^j (u(x) + \ln x_n)| \leq C_j r_1^{1-j} \quad\text{for }j\geq 1, \, x\in B(y,r_1/2).
	\end{align*}
	Taking $j=1$, we therefore see that $|\nabla(u(x) + \ln x_n)| \leq C_1$ in $B(y,r_1/2)$, and since $B(y, r_1) \subset \{|x'-x_0'| < R,\,0<x_n < \widetilde{r}\}$ was arbitrary, we therefore have $|\nabla(u(x) + \ln x_n)| \leq C_1$ in $\{|x'-x_0'| < R,\,0<x_n < \widetilde{r}\}$. Converting this back into a statement for $w$, we see
	\begin{align*}
	\bigg|\nabla \ln \bigg(\frac{w}{x_n}\bigg)\bigg| \leq C \quad \text{in }\{|x'-x_0'| < R,\,0<x_n < \widetilde{r}\}.
	\end{align*}
	By Lemmas \ref{l1} and \ref{l2}, this implies $|\nabla w - \nabla x_n| \leq Cx_n$ for $x_n$ small, and the result follows. 
\end{proof}

\subsection{Symmetry via moving spheres}\label{sym}

Having established Lemmas \ref{l1}, \ref{l2} and \ref{17}, it is clear that Proposition \ref{A'} will follow once we show that $w = w(x_n)$. As discussed in the introduction, for this step we apply a variant of the method of moving spheres. An important property of \eqref{15} that will be used is its conformal invariance: if $\phi$ is a M\"obius transformation on $\mathbb{R}^n\cup \{\infty\}$ and $\Omega\subset\mathbb{R}^n$ is such that $\phi|_\Omega:\Omega\rightarrow\mathbb{R}_+^n$ is a bijection, then $w_\phi \defeq |J_\phi|^{-1/n}(w\circ\phi)$ (where $|J_\phi|$ is the Jacobian determinant of $\phi$) satisfies\footnote{If $\Omega$ is a ball or $\Omega$ is a half-space with $\phi(\infty) = \infty$, then \eqref{15} and \eqref{15'} are equivalent. If $\Omega$ is a half-space and $\phi(\infty)\not=\infty$, then \eqref{15'} only implies \eqref{15} if $w_\phi$ satisfies certain behaviour at infinity.}
\begin{align}\label{15'}
\begin{cases}
f(\lambda(-A_{w_\phi})) = \frac{1}{2}, \quad \lambda(-A_{w_\phi})\in\Gamma & \text{in }\Omega \\
w_\phi = 0 & \text{on }\phi^{-1}(\partial\mathbb{R}_+^n) = \partial\Omega\backslash \phi^{-1}(\infty).
\end{cases}
\end{align}
We recall that Li \& Li \cite{LL03} characterised conformally invariant second order differential operators as those of the form $F[w] = f(\lambda(A_w))$, where $f:\mathbb{R}^n\rightarrow\mathbb{R}$ is a symmetric function.\medskip 

We now describe the set-up for our application of the method of moving spheres, during which the reader may wish to refer to Figure 1. Fix a unit vector $\nu\in \mathbb{S}_+^{n-1}\backslash\{(0,\dots,0,1)\} = \{x\in\mathbb{S}^{n-1}: 0<x_n<1\}$ and define the associated hyperplane through the origin $H_\nu = \{x\in\mathbb{R}^n: x\cdot \nu = 0 \}$. For $R>0$, we denote by $S_{\nu,R} = \{x\in\mathbb{R}^n: |x+R\nu| = R\}$ the sphere of radius $R$ centred at $-R\nu$, $\widetilde{D}_{\nu,R} =\{x\in\mathbb{R}^n: |x+R\nu| < R,\, x_n>0\}$ the associated solid spherical cap in the upper half-space, and $D_{\nu,R}$ the inversion of $\widetilde{D}_{\nu,R}$ across $S_{\nu,R}$, that is
\begin{align*}
x \in D_{\nu,R} \iff x_{\nu,R} \defeq \frac{R^2}{|x+R\nu|^2}(x+R\nu) - R\nu \in \widetilde{D}_{\nu,R}. 
\end{align*}
Finally, for $x\in D_{\nu,R}$ we define
\begin{align*}
w_{\nu,R}(x) \defeq \frac{|x+R\nu|^2}{R^2}w(x_{\nu,R}).
\end{align*}
Note that the factor $ \frac{|x+R\nu|^2}{R^2}$ is the $n$'th root of the Jacobian of the transformation $x_{\nu,R}\mapsto x$. We also point out that $w_{\nu,R}^{-\frac{n-2}{2}}$ is the Kelvin transformation of $w^{-\frac{n-2}{2}}$ across the sphere $S_{\nu,R}$. By conformal invariance, $A_{w_{\nu,R}}(x) = O^t A_w(x_{\nu,R})O$ for some $O\in O(n)$, and it follows from this fact and \eqref{15} that
\begin{align*}
\begin{cases}
f(\lambda(-A_{w_{\nu,R}})) = \frac{1}{2} & \text{in }D_{\nu,R} \\
w_{\nu,R} = w & \text{on }S_{\nu,R} \cap \{x_n \geq 0\} \\
w_{\nu,R} = 0 & \text{on } \partial D_{\nu,R} \backslash S_{\nu,R}. 
\end{cases}
\end{align*}

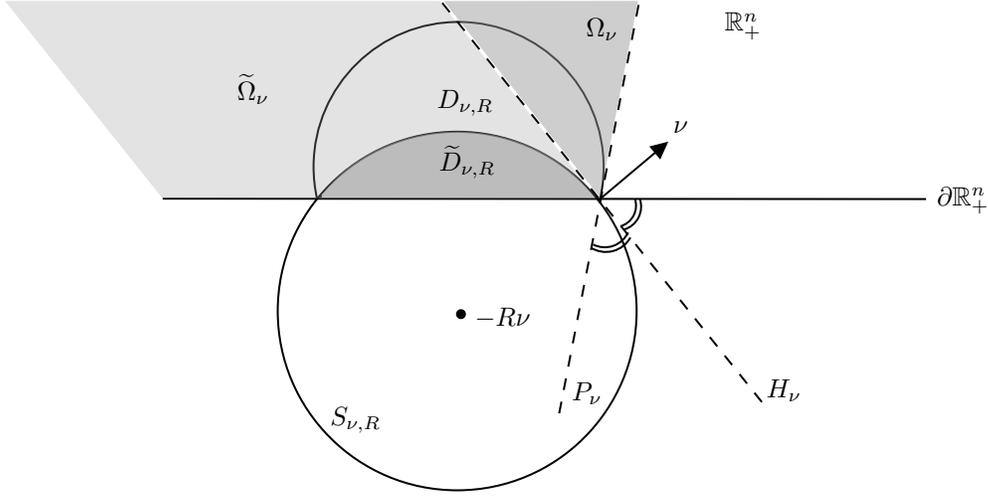
\begin{figure}

	\tikzset{every picture/.style={line width=0.75pt}} 
	
	\begin{tikzpicture}[x=0.75pt,y=0.75pt,yscale=-1,xscale=1]
	uncomment if require: \path (30,336); 
	
	\draw  [line width=0.75]  (224,222.5) .. controls (224,172.52) and (264.52,132) .. (314.5,132) .. controls (364.48,132) and (405,172.52) .. (405,222.5) .. controls (405,272.48) and (364.48,313) .. (314.5,313) .. controls (264.52,313) and (224,272.48) .. (224,222.5) -- cycle ;
	\draw  [draw opacity=0][line width=0.75]  (243.87,165.92) .. controls (242.65,160.51) and (242.03,154.87) .. (242.08,149.09) .. controls (242.44,108.69) and (275.47,76.24) .. (315.87,76.6) .. controls (356.26,76.96) and (388.72,110) .. (388.36,150.39) .. controls (388.31,155.86) and (387.66,161.18) .. (386.48,166.29) -- (315.22,149.74) -- cycle ; \draw  [line width=0.75]  (243.87,165.92) .. controls (242.65,160.51) and (242.03,154.87) .. (242.08,149.09) .. controls (242.44,108.69) and (275.47,76.24) .. (315.87,76.6) .. controls (356.26,76.96) and (388.72,110) .. (388.36,150.39) .. controls (388.31,155.86) and (387.66,161.18) .. (386.48,166.29) ;  
	\draw    (165.98,165.79) -- (551,166) ;
	\draw  [dash pattern={on 4.5pt off 4.5pt}]  (307,67) -- (471,272) ;
	\draw    (385.33,166.99) -- (418.7,138.93) ;
	\draw [shift={(421,137)}, rotate = 139.95] [fill={rgb, 255:red, 0; green, 0; blue, 0 }  ][line width=0.08]  [draw opacity=0] (8.93,-4.29) -- (0,0) -- (8.93,4.29) -- cycle    ;
	\draw  [fill={rgb, 255:red, 0; green, 0; blue, 0 }  ,fill opacity=1 ] (314.43,224.07) .. controls (314.43,222.97) and (315.32,222.07) .. (316.43,222.07) .. controls (317.53,222.07) and (318.43,222.97) .. (318.43,224.07) .. controls (318.43,225.18) and (317.53,226.07) .. (316.43,226.07) .. controls (315.32,226.07) and (314.43,225.18) .. (314.43,224.07) -- cycle ;
	\draw  [fill={rgb, 255:red, 155; green, 155; blue, 155 }  ,fill opacity=0.54 ][dash pattern={on 0.75pt off 750pt}] (245.42,165.07) .. controls (259.48,145.26) and (285.16,132) .. (314.5,132) .. controls (343.89,132) and (369.6,145.3) .. (383.66,165.17) -- cycle ;
	\draw  [dash pattern={on 4.5pt off 4.5pt}]  (406,68) -- (364.94,279.41) ;
	\draw  [draw opacity=0][fill={rgb, 255:red, 155; green, 155; blue, 155 }  ,fill opacity=0.28 ] (86.46,66) -- (305.65,66) -- (385.17,165.79) -- (165.98,165.79) -- cycle ;
	\draw  [draw opacity=0][fill={rgb, 255:red, 155; green, 155; blue, 155 }  ,fill opacity=0.49 ] (386.46,165) -- (307.69,66.04) -- (405.68,65.98) -- cycle ;
	\draw  [draw opacity=0] (404.15,165.82) .. controls (405.05,168.52) and (404.9,171.58) .. (403.48,174.31) .. controls (402.06,177.03) and (399.64,178.9) .. (396.91,179.7) -- (393.85,169.28) -- cycle ; \draw   (404.15,165.82) .. controls (405.05,168.52) and (404.9,171.58) .. (403.48,174.31) .. controls (402.06,177.03) and (399.64,178.9) .. (396.91,179.7) ;  
	\draw  [draw opacity=0] (406.65,165.95) .. controls (407.66,168.96) and (407.49,172.36) .. (405.91,175.4) .. controls (404.33,178.43) and (401.64,180.52) .. (398.59,181.41) -- (395.19,169.8) -- cycle ; \draw   (406.65,165.95) .. controls (407.66,168.96) and (407.49,172.36) .. (405.91,175.4) .. controls (404.33,178.43) and (401.64,180.52) .. (398.59,181.41) ;  
	\draw  [draw opacity=0] (401.88,185.39) .. controls (400.07,188.65) and (396.98,191.2) .. (393.08,192.23) .. controls (389.19,193.25) and (385.25,192.56) .. (382.07,190.62) -- (389.46,178.47) -- cycle ; \draw   (401.88,185.39) .. controls (400.07,188.65) and (396.98,191.2) .. (393.08,192.23) .. controls (389.19,193.25) and (385.25,192.56) .. (382.07,190.62) ;  
	\draw  [draw opacity=0] (400.71,182.82) .. controls (399.29,185.91) and (396.66,188.43) .. (393.19,189.62) .. controls (389.73,190.82) and (386.11,190.46) .. (383.08,188.92) -- (388.97,177.39) -- cycle ; \draw   (400.71,182.82) .. controls (399.29,185.91) and (396.66,188.43) .. (393.19,189.62) .. controls (389.73,190.82) and (386.11,190.46) .. (383.08,188.92) ;  
	
	\draw (304,137.4) node [anchor=north west][inner sep=0.75pt]  [font=\footnotesize]  {$\widetilde{D}_{\nu ,R}$};
	\draw (303,109.4) node [anchor=north west][inner sep=0.75pt]  [font=\footnotesize]  {$D_{\nu ,R}$};
	\draw (469,255.4) node [anchor=north west][inner sep=0.75pt]  [font=\footnotesize]  {$H_{\nu }$};
	\draw (422,125.4) node [anchor=north west][inner sep=0.75pt]  [font=\footnotesize]  {$\nu $};
	\draw (322,219.4) node [anchor=north west][inner sep=0.75pt]  [font=\footnotesize]  {$-R\nu $};
	\draw (448,70.4) node [anchor=north west][inner sep=0.75pt]  [font=\footnotesize]  {$\mathbb{R}_{+}^{n}$};
	\draw (555,158.4) node [anchor=north west][inner sep=0.75pt]  [font=\footnotesize]  {$\partial \mathbb{R}_{+}^{n}$};
	\draw (249,269.4) node [anchor=north west][inner sep=0.75pt]  [font=\footnotesize]  {$S_{\nu ,R}$};
	\draw (370.92,257.08) node [anchor=north west][inner sep=0.75pt]  [font=\footnotesize]  {$P_{\nu }$};
	\draw (202,101.4) node [anchor=north west][inner sep=0.75pt]  [font=\footnotesize]  {$\widetilde{\Omega }_{\nu }$};
	\draw (378,73.4) node [anchor=north west][inner sep=0.75pt]  [font=\footnotesize]  {$\Omega _{\nu }$};

	\end{tikzpicture}
	\caption{\small{The set-up for the method of moving spheres. Note that since $P_\nu$ is the reflection of $\{x_n=0\}$ across $H_\nu$, the two angles indicated in the figure are equal.}}
\end{figure}

Our goal is to show $w \geq w_{\nu,R}$ in $D_{\nu,R}$; let us explain now why this implies $w=w(x_n)$ on $\mathbb{R}_+^n$. Note that the sphere containing $\partial D_{\nu,R}\backslash S_{\nu,R}$ is the inversion of the hyperplane $\{x_n=0\}$ about the sphere $S_{\nu,R}$. By conformality of the inversion map, it follows that the tangent hyperplane $P_\nu$ to $\partial D_{\nu,R}\backslash S_{\nu,R}$ at the origin is the reflection of the hyperplane $\{x_n=0\}$ about the hyperplane $H_\nu$, and is therefore independent of $R$. Let $\widetilde{\Omega}_\nu = \cup_{R>0}\widetilde{D}_{\nu, R}$ (which is also equal to $ \{x\in\mathbb{R}^n_+:x\cdot\nu<0\}$) and let $\Omega_\nu$ denote the reflection of $\widetilde{\Omega}_\nu$ across $H_\nu$. Note that $\Omega_\nu = \cup_R (D_{\nu,R}\cap \Omega_\nu)$ is an increasing union in $R$ and hence any fixed $x\in \Omega_\nu$ belongs to $D_{\nu,R}$ for all sufficiently large $R$. Moreover, if $x\in \Omega_\nu$, then $x_{\nu,R}\rightarrow x - 2(x\cdot \nu)\nu$ as $R\rightarrow\infty$, that is $x_{\nu,R}$ tends to the reflection of $x$ across the hyperplane $H_\nu$. Therefore, for any $x\in \Omega_\nu$, 
\begin{align*}
\lim_{R\rightarrow+\infty}w_{\nu,R}(x)= w(x-2(x\cdot \nu)\nu).
\end{align*}
Since we assume we have shown $w \geq w_{\nu,R}$ in $D_{\nu,R}$, we therefore obtain 
\begin{align*}
w(x) \geq  w(x-2(x\cdot \nu)\nu) \quad \text{ in }\Omega_\nu \text{ for all } \nu\in\mathbb{S}^{n-1}_+\backslash\{(0,\dots,0,1)\}. 
\end{align*}
By continuity, it follows that this inequality remains true for $\nu\in\partial\mathbb{S}_+^{n-1}$, that is
\begin{align}\label{152}
w(x) \geq w(x-2(x\cdot \nu)\nu) \quad\text{in }\{x\in\mathbb{R}_+^n: x\cdot \nu>0\}  \text{ for all } \nu\in\partial \mathbb{S}^{n-1}_+. 
\end{align}
Now for any two distinct points $x=(x', x_n)$ and $\widetilde{x}=(\widetilde{x}', x_n)$ in $\mathbb{R}_+^n$ with $|x'| = |\widetilde{x}'|$, for $\nu = \frac{x-\widetilde{x}}{|x-\widetilde{x}|}$ we have $x - 2(x\cdot\nu)\nu = \widetilde{x}$, and thus $w(x) \geq w(\widetilde{x})$ by \eqref{152}. Reversing the roles of $x$ and $\widetilde{x}$, we therefore see that $w(x) = w(\widetilde{x})$. Finally, for any two distinct points $x=(x', x_n)$ and $\widetilde{x}=(\widetilde{x}', x_n)$, we may replace the origin in the above argument with any other point $(\hat{x}',0)$ such that $|x'-\hat{x}'| = |\widetilde{x}' - \hat{x}'|$ to again conclude that $w(x) = w(\widetilde{x})$, as required.\medskip

Thus, to complete the proof of Theorem \ref{A}, we need to show:

\begin{lem}\label{19}
	For any $\nu\in\mathbb{S}_+^{n-1}\backslash\{(0,\dots,0,1)\}$ and $R>0$, $w \geq w_{\nu,R}$ in $D_{\nu,R}$. 
\end{lem}

Note that neither of the comparison principles in Propositions \ref{800} and \ref{510} can be applied directly to obtain Lemma \ref{19}. Indeed, Proposition \ref{800} doesn't apply since we do not have strict inequality between $w_{\nu,R}$ and $w$ on $S_{\nu,R}\cap \{x_n\geq 0\}$, and Proposition \ref{510} doesn't apply since $w = w_{\nu,R} = 0$ on $\partial D_{\nu,R} \cap \{x_n=0\}$. To bypass these issues, we will first use the bounds and regularity for $w$ obtained in the previous section to show directly that $w \geq w_{\nu,R}$ for points in $D_{\nu,R}$ sufficiently close to $\{x_n=0\}$. To state this result, let 
\begin{align*}
U = \{\mu\in\mathbb{S}^{n-1}:\mu_n > \nu_n\}
\end{align*}
and note that any $x\in D_{\nu,R}$ (resp.~$x\in\widetilde{D}_{\nu,R}$) can be written as $x = -R\nu  + r\mu$ for some $\mu\in U$ and $r > R$ (resp.~$\mu\in U$ and $r\in(r_\mu, R)$, where $r_\mu$ is such that $(-R\nu  + r_\mu \mu)_n =0$). We prove: 

\begin{lem}\label{16}
	For any $\nu\in\mathbb{S}_+^{n-1}\backslash\{(0,\dots,0,1)\}$ and $R>0$, there exists $s_0 = s_0(\nu,R,w)>0$ such that $w \geq w_{\nu,R}$ in $\Sigma_{s_0} \defeq \{x = -R\nu  + r\mu\in D_{\nu,R}: \mu\in U,\, \operatorname{dist}(\mu,\partial U)<s_0 \text{ and } r \leq R+s_0\}$.
\end{lem}

\begin{proof}
	For each $\mu\in U$, define $\gamma_\mu(r) = \frac{1}{r}w(-R\nu  + r\mu)$ on the maximal domain $\{r \geq r_\mu\}$, where $r_\mu$ is as defined before the statement of the lemma. We claim that there exists $s_0>0$ such that if $\operatorname{dist}(\mu,\partial U) <s_0$, then $\gamma_\mu$ is differentiable with $\gamma_\mu'>0$ in $[r_\mu,R+s_0]$. \medskip 
	
	To prove the claim, we start with the following observation: there is a positive function $\eta = \eta(s)$, monotonically decreasing to zero as $s\rightarrow 0$, such that 
	\begin{align}\label{509}
	\mu\in U, \,\,\,\, \operatorname{dist}(\mu,\partial U)\leq s \,\,\,\, \text{and} \,\,\,\, r\in[r_\mu, R+s] \,\,\, \implies \,\,\,  (-R\nu +r\mu)_n \leq \eta(s). 
	\end{align}
	
	\begin{figure}

	\tikzset{every picture/.style={line width=0.75pt}} 
	
	\begin{tikzpicture}[x=0.75pt,y=0.75pt,yscale=-1,xscale=1]
	uncomment if require: \path (0,359); 
	
	\draw    (172,260) -- (521,260) ;
	\draw  [draw opacity=0] (228.62,259.99) .. controls (202.51,234.88) and (186.33,199.53) .. (186.54,160.44) .. controls (186.93,84.73) and (248.63,23.68) .. (324.34,24.07) .. controls (400.05,24.47) and (461.11,86.17) .. (460.71,161.88) .. controls (460.51,200.62) and (444.26,235.52) .. (418.29,260.32) -- (323.63,161.16) -- cycle ; \draw   (228.62,259.99) .. controls (202.51,234.88) and (186.33,199.53) .. (186.54,160.44) .. controls (186.93,84.73) and (248.63,23.68) .. (324.34,24.07) .. controls (400.05,24.47) and (461.11,86.17) .. (460.71,161.88) .. controls (460.51,200.62) and (444.26,235.52) .. (418.29,260.32) ;  
	\draw  [draw opacity=0] (228.57,259.72) .. controls (244.45,218.42) and (281.56,189.58) .. (324.57,189.92) .. controls (367.15,190.25) and (403.5,219.09) .. (419,260) -- (323.67,304.48) -- cycle ; \draw   (228.57,259.72) .. controls (244.45,218.42) and (281.56,189.58) .. (324.57,189.92) .. controls (367.15,190.25) and (403.5,219.09) .. (419,260) ;  
	\draw  [dash pattern={on 7.5pt off 1.5pt}]  (228.62,259.99) -- (325,298) ;
	\draw  [dash pattern={on 7.5pt off 1.5pt}]  (418.29,260.32) -- (325,298) ;
	\draw  [dash pattern={on 0.75pt off 3pt}]  (457,184) -- (325,298) ;
	\draw  [fill={rgb, 255:red, 0; green, 0; blue, 0 }  ,fill opacity=1 ] (323,298) .. controls (323,296.9) and (323.9,296) .. (325,296) .. controls (326.1,296) and (327,296.9) .. (327,298) .. controls (327,299.1) and (326.1,300) .. (325,300) .. controls (323.9,300) and (323,299.1) .. (323,298) -- cycle ;
	\draw  [draw opacity=0][fill={rgb, 255:red, 155; green, 155; blue, 155 }  ,fill opacity=0.52 ] (408,227) .. controls (417,221) and (462,173) .. (459,187) .. controls (456,201) and (454,211) .. (444,226) .. controls (434,241) and (420,268) .. (417,255) .. controls (414,242) and (399,233) .. (408,227) -- cycle ;
	\draw    (326,297) -- (351.72,274.95) ;
	\draw [shift={(354,273)}, rotate = 139.4] [fill={rgb, 255:red, 0; green, 0; blue, 0 }  ][line width=0.08]  [draw opacity=0] (8.93,-4.29) -- (0,0) -- (8.93,4.29) -- cycle    ;
	\draw    (466,186) -- (466,257) ;
	\draw [shift={(466,259)}, rotate = 270] [color={rgb, 255:red, 0; green, 0; blue, 0 }  ][line width=0.75]    (10.93,-3.29) .. controls (6.95,-1.4) and (3.31,-0.3) .. (0,0) .. controls (3.31,0.3) and (6.95,1.4) .. (10.93,3.29)   ;
	\draw [shift={(466,184)}, rotate = 90] [color={rgb, 255:red, 0; green, 0; blue, 0 }  ][line width=0.75]    (10.93,-3.29) .. controls (6.95,-1.4) and (3.31,-0.3) .. (0,0) .. controls (3.31,0.3) and (6.95,1.4) .. (10.93,3.29)   ;

	\draw (308,145.4) node [anchor=north west][inner sep=0.75pt]  [font=\small]  {$D_{\nu, R}$};
	\draw (308,205.4) node [anchor=north west][inner sep=0.75pt]  [font=\small]  {$\widetilde{D}_{\nu, R}$};
	\draw (419,219.4) node [anchor=north west][inner sep=0.75pt]  [font=\small]  {$\Sigma _{s}$};
	\draw (472,214.4) node [anchor=north west][inner sep=0.75pt]  [font=\small]  {$\eta ( s)$};
	\draw (307,303.4) node [anchor=north west][inner sep=0.75pt]  [font=\small]  {$-R\nu $};
	\draw (333,264.4) node [anchor=north west][inner sep=0.75pt]  [font=\small]  {$\mu $};

	\end{tikzpicture}
		\vspace*{-10mm}\caption{\small{The cone with vertex at $-Rv$ indicated in the figure is generated by vectors in $\partial U$. If $s$ is sufficiently small and the vector $\mu$ satisfies $\operatorname{dist}(\mu,\partial U)\leq s$, then $\Sigma_s$ is the shaded region in the figure. As $s$ decreases, the maximum height $\eta(s)$ of a point in $\Sigma_s$ decreases.}}
	\end{figure}
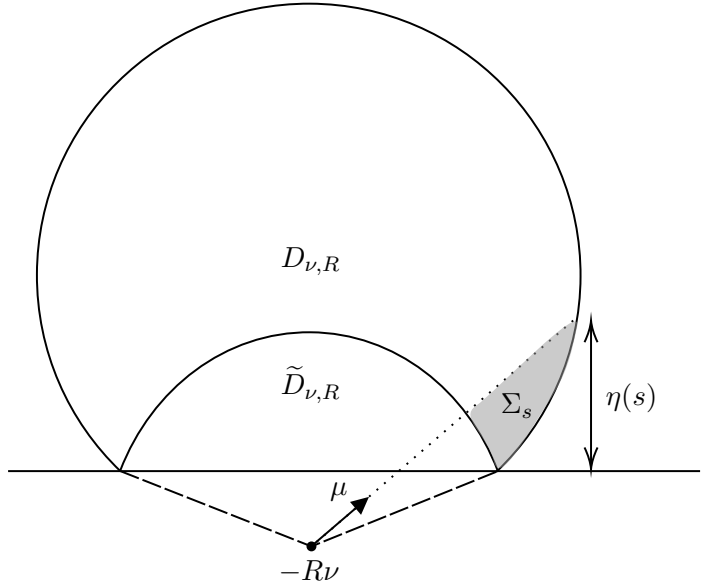

\noindent (For instance, just take $\eta(s) = \sup_{\operatorname{dist}(\mu,\partial U)=s}(-R\nu+(R+s)\mu)_n$ -- see Figure 2). By the $C^1$ regularity near $\{x_n=0\}$ asserted in Proposition \ref{A'} and the estimates in Lemma \ref{17}, it follows from \eqref{509} that there exists small $s_0>0$ such that if $\mu\in U$ and $\operatorname{dist}(\mu,\partial U)<s_0$, then $\gamma_\mu$ is differentiable in $[r_\mu, R+s_0]$ with
	\begin{align*}
	\gamma_\mu'(r) & = -\frac{1}{r^2}w(-R\nu +r\mu) + \frac{1}{r}\nabla w(-R\nu +r\mu)\cdot \mu \nonumber \\
	& = -\frac{1}{r^2}\Big((-R\nu +r\mu)_n(1 + o(1))\Big) + \frac{1}{r}\Big((0,\dots,0,1)\cdot \mu + o(1)\Big) \nonumber \\
	& = \frac{R\nu_n}{r^2} + \bigg(\frac{R}{r^2} + \frac{1}{r}\bigg)o(1)>0
	\end{align*}
	for $r\in[r_\mu, R+s_0]$, where the $o(1)$ terms can be made arbitrarily small by shrinking $s_0$. This completes the proof of the claim.\medskip

	We now continue with the proof of the lemma. First note that if $x = -R\nu  + r\mu \in D_{\nu,R}$, then $\mu\in U$ and $\frac{R^2}{r}\geq r_\mu$ since
	\begin{align*}
	-R\nu  + \frac{R^2}{r}\mu = \frac{R^2}{|x+R\nu |^2}(x+R\nu ) - Re = x_{\nu,R} \in \widetilde{D}_{\nu,R}.
	\end{align*} 
	Therefore, if $x = -R\nu  + r\mu \in\Sigma_{s_0}$ where $s_0$ is as in the above claim, then $r_\mu \leq \frac{R^2}{r}\leq r \leq R+s_0$ and hence 
	\begin{align}\label{507}
	\gamma_\mu\bigg(\frac{R^2}{r}\bigg) \leq \gamma_\mu(r). 
	\end{align}  
	Since
	\begin{align*}
	\gamma_\mu\bigg(\frac{R^2}{r}\bigg) =  \frac{r}{R^2}w\bigg(\!\!-\!R\nu + \frac{R^2}{r}\mu\bigg) = \frac{|x+R\nu |}{R^2}w(x_{\nu,R})
	\end{align*} 
	and
	\begin{align*}
	\gamma_\mu(r) = \frac{1}{r}w(-R\nu +r\mu) = \frac{1}{|x+R\nu |}w(x),
	\end{align*}
	we see that \eqref{507} is equivalent to $w_{\nu,R}(x) \leq w(x)$. This completes the proof of the lemma.
\end{proof}

\begin{proof}[Proof of Lemma \ref{19}]
	By definition of $w$ and $w_{\nu,R}$, and by Lemma \ref{16}, we have $w>0$ and $w \geq w_{\nu,R}$ on $\partial(D_{\nu,R}\backslash \Sigma_s)$ for some $s>0$ fixed sufficiently small. It follows that for each $\beta>1$, $\beta w > w_{\nu,R}$ on $\partial(D_{\nu,R}\backslash \Sigma_s)$. Since 
	\begin{align*}
	f(\lambda(-A_{\beta w})) = \beta^2 f(\lambda(-A_w)) = \frac{\beta^2}{2} > \frac{1}{2} = f(\lambda(-A_{w_{\nu,R}})) \quad \text{in } D_{\nu,R}\backslash \Sigma_s,
	\end{align*}
	we may therefore apply the comparison principle in Proposition \ref{800} to obtain $\beta w > w_{\nu,R}$ in $D_{\nu,R}\backslash \Sigma_s$. Taking $\beta\rightarrow 1$, we obtain $w \geq w_{\nu,R}$ in $D_{\nu,R}\backslash \Sigma_s$ and hence, by Lemma \ref{16}, in $D_{\nu,R}$. This completes the proof of Lemma \ref{19}.
\end{proof}

\section{ODE analysis and proofs of Theorems \ref{B}, \ref{A} and \ref{D}}\label{s3}

As a consequence of Proposition \ref{A'}, under the assumptions \eqref{21''}--\eqref{23''}, \eqref{24'} and \eqref{25'}, any continuous viscosity solution to \eqref{15} is of the form $w=w(x_n)$. In this section we carry out the ODE analysis of such solutions to prove Theorems \ref{B}, \ref{A} and \ref{D}. 

\subsection{Preliminaries and set-up for ODE analysis}

Let us denote the $x_n$ variable in $\mathbb{R}^n_+$ by $t$, which we may refer to as the time variable. If $u=u(t)$, then $-A_u = \operatorname{diag}(\lambda_1^u, \lambda_2^u, \dots,\lambda_2^u)$ where
\begin{align*}
\lambda_1^u = -u u'' + \frac{1}{2}(u')^2 \quad \text{and} \quad \lambda_2^u = \frac{1}{2}(u')^2.
\end{align*}
In light of Proposition \ref{A'}, under the assumptions \eqref{21''}--\eqref{23''}, \eqref{24'} and \eqref{25'}, $w$ is a continuous viscosity solution to \eqref{15} if and only if $w$ is a function of $t$ satisfying
\begin{align}\label{1}
f(\lambda_1^w,\lambda_2^w,\dots,\lambda_2^w) = \frac{1}{2}, \quad (\lambda_1^w,\lambda_2^w,\dots,\lambda_2^w)\in\Gamma
\end{align}
in the viscosity sense and $w(0) = 0$.\medskip 

At this point it is desirable to express $\lambda_1^w$ (which contains $w''$) as a function $\psi$ of $\lambda_2^w$ (which only contains $w'$). Observe that under the general symmetry and ellipticity conditions \eqref{21''}--\eqref{23''}, \eqref{24'} and \eqref{25'}: \smallskip
\begin{enumerate}[$\bullet$]
	\item For each $s\in\mathbb{R}$, there exists at most one value $\psi(s)$ such that \begin{align*}
	f(\psi(s),s,\dots,s)=\frac{1}{2}. 
	\end{align*} 
	\item If $s_1<s_2$ are such that both $\psi(s_1)$ and $\psi(s_2)$ exist, then for each $s\in(s_1,s_2)$
	\begin{align}\label{511}
	f(\psi(s_2),s,\dots,s)<\frac{1}{2}<f(\psi(s_1), s, \dots, s).
	\end{align}
	Therefore (by the intermediate value theorem) $\psi$ exists and is decreasing on the interval $[s_1,s_2]$. In particular, the domain of $\psi$, denoted $\operatorname{Dom}(\psi)$, is connected\footnote{In \eqref{511}, the right inequality is understood in the usual sense since $(\psi(s_1),s,\dots,s)\in\Gamma$ by \eqref{22''}. The left inequality is understood in the sense that either $(\psi(s_2), s,\dots,s)\in\mathbb{R}^n\backslash\overline{\Gamma}$, or $(\psi(s_2), s,\dots,s)\in\overline{\Gamma}$ with $f(\psi(s_2), s, \dots, s)<\frac{1}{2}$.}.\smallskip
	\item $\operatorname{Dom}(\psi)$ is open (by the implicit function theorem) and is therefore an open interval containing $s=1$.\medskip
\end{enumerate}

To simplify matters, we now replace the very general assumptions \eqref{21''}--\eqref{23''} by the more standard assumptions \eqref{21'}--\eqref{23'}. In this case, $\mu_\Gamma^+$ is well-defined and a more concrete description of $\operatorname{Dom}(\psi)$ is available. Indeed, define
\begin{align}\label{139}
\eta \defeq \lim_{R\rightarrow+\infty} f(R,1,\dots,1)\in(1,+\infty].
\end{align}
Since
\begin{align*}
f(-\mu_\Gamma^+, 1,\dots,1) = 0,
\end{align*} 
it follows that for each $s\in(\frac{1}{2\eta},\infty)$ there exists a unique number $\phi(s) =s^{-1}\psi(s) \in(-\mu_\Gamma^+, \infty)$ such that 
\begin{align*}
f(\phi(s),1,\dots,1) = \frac{1}{2s}.
\end{align*}
Clearly $\operatorname{Dom}(\psi) \supseteq \operatorname{Dom}(\phi) = (\frac{1}{2\eta},\infty)$ and $\phi:(\frac{1}{2\eta},\infty)\rightarrow (-\mu_\Gamma^+, \infty)$ is a smooth decreasing bijection satisfying $\phi(\frac{1}{2})=1$ (by \eqref{25'}). In particular, $\lim_{s\rightarrow\frac{1}{2\eta}}\phi(s) = \infty$ and $\lim_{s\rightarrow\infty}\phi(s) = -\mu_\Gamma^+$, and $\phi$ restricts to a bijection from $[\frac{1}{2},\infty)$ onto $(-\mu_\Gamma^+, 1]$. We have: 

\begin{lem}\label{140}
	Suppose $(f,\Gamma)$ satisfies \eqref{21'}--\eqref{25'}. Then either: \smallskip
	\begin{enumerate}
		\item $\operatorname{Dom}(\psi) = \operatorname{Dom}(\phi) = (\frac{1}{2\eta},\infty)$, or\smallskip
		\item $\operatorname{Dom}(\psi) = \mathbb{R}$.\smallskip
	\end{enumerate}
Moreover, if $\eta<\infty$ then the first assertion is true, and if $(1,0,\dots,0)\in\Gamma$ then $\eta = \infty$ and the second assertion is true. 
\end{lem}

\begin{proof}
	The proof is routine, and we refer the reader to Appendix \ref{appa} for the details.
\end{proof}

Let us now return to the equation \eqref{1}. At points where $w$ is smooth, \eqref{1} is equivalent to 
\begin{align}\label{2'}
\lambda_1^w = \psi(\lambda_2^w) \quad \text{and} \quad  \lambda_2^w \in \operatorname{Dom}(\psi),
\end{align}
that is
\begin{align}\label{2}
ww'' = \frac{1}{2}(w')^2 - \psi\Big(\frac{1}{2}(w')^2\Big)\quad \text{and} \quad \frac{1}{2}(w')^2 \in \operatorname{Dom}(\psi).
\end{align}

\begin{defn}\label{138}
	Let $I\subset\mathbb{R}$ be an open interval. We call $w_1\in C^0(I)$ a viscosity supersolution to \eqref{2'} if for any $t_0\in I$ and $u\in C^2(I)$ satisfying $u(t_0) = w_1(t_0)$ and $u \leq w_1$ near $t_0$,
	\begin{align}\label{126}
	\lambda_2^u(t_0) \in \operatorname{Dom}(\psi)\quad \text{and} \quad \lambda_1^u(t_0) \geq \psi(\lambda_2^u(t_0)). 
	\end{align}
	We call $w_2\in C^0(I)$ a viscosity subsolution to \eqref{2'} if for any $t_0\in I$ and $u\in C^2(I)$ satisfying $u(t_0) = w_2(t_0)$ and $u \geq w_2$ near $t_0$, either 
	\begin{align}\label{135}
	\lambda_2^u(t_0)  \not\in \operatorname{Dom}(\psi) 
	\end{align}
	or 
	\begin{align}\label{136}
	\lambda_2^u(t_0) \in \operatorname{Dom}(\psi) \quad \text{and} \quad \lambda_1^u(t_0) \leq \psi(\lambda_2^u(t_0)). 
	\end{align}
	We call $w\in C^0(I)$ a viscosity solution to \eqref{2'} if it is both a viscosity supersolution and a viscosity subsolution.
\end{defn}

\begin{lem}\label{600}
	Suppose that $(f,\Gamma)$ satisfies \eqref{21'}--\eqref{25'} and let $I\subset(0,\infty)$ be a non-empty open interval. Then $w=w(x_n)$ is a viscosity supersolution (resp.~viscosity subsolution) to \eqref{1} on $I$ (in the sense of Definition \ref{137}) if and only if it is a viscosity supersolution (resp.~viscosity subsolution) to \eqref{2'} on $I$ (in the sense of Definition \ref{138}). 
\end{lem}

\begin{proof}
		The proof is routine, and we refer the reader to Appendix \ref{appa} for the details.
\end{proof}

\begin{lem}\label{p0}
	Suppose that $(f,\Gamma)$ satisfies \eqref{21'}--\eqref{25'} and $w\in C^0([0,\infty))$ satisfies $w(0) = 0$. Then $w$ is a continuous viscosity solution to \eqref{1} if and only if $w\in C^\infty_{\operatorname{loc}}((0,\infty))\cap C^1_{\operatorname{loc}}([0,\infty))$ and $w$ satisfies \eqref{2} in the classical sense. 
\end{lem}
\begin{proof}
	It is clear by definition of $\psi$ that if $w\in C^\infty_{\operatorname{loc}}((0,\infty))\cap C^1_{\operatorname{loc}}([0,\infty))$ satisfies \eqref{2}, then it is a smooth solution (and hence continuous viscosity solution) to \eqref{1}. Conversely, suppose that $w$ is a continuous viscosity solution to \eqref{1} satisfying $w(0) = 0$. By Proposition \ref{A'}, $w(x_n) \geq x_n$. On the other hand, by Lemma \ref{600}, $w$ is a continuous viscosity solution to \eqref{2}, which is a uniformly elliptic equation on the set $\{w \geq \ep\}$ for any $\ep>0$. Classical regularity theory for viscosity solutions (see e.g.~\cite{CC95}) therefore implies $w\in C^\infty_{\operatorname{loc}}((0,\infty))$, and hence $w$ satisfies \eqref{2} in the classical sense. The fact that $w$ is $C^1$ up to $t=0$ follows from Proposition \ref{A'}. 
\end{proof}

In fact, instead of working with \eqref{2}, it will be more convenient to work with the corresponding equation involving $\phi$:
\begin{align}\label{700}
ww'' = \frac{1}{2}(w')^2\bigg[1 - \phi\bigg(\frac{1}{2}(w')^2\bigg)\bigg], \quad \frac{1}{2}(w')^2 \in \operatorname{Dom}(\phi). 
\end{align}
We have the following counterpart to Lemma \ref{p0}:
\begin{lem}\label{p0'}
	Suppose that $(f,\Gamma)$ satisfies \eqref{21'}--\eqref{25'} and $w\in C^0([0,\infty))$ satisfies $w(0) = 0$. Then $w$ is a continuous viscosity solution to \eqref{1} if and only if $w\in C^\infty_{\operatorname{loc}}((0,\infty))\cap C^1_{\operatorname{loc}}([0,\infty))$ and $w$ satisfies \eqref{700} in the classical sense. 
\end{lem}
\begin{proof}
	By Lemma \ref{p0}, it suffices to show that $w$ is a smooth solution to \eqref{2} if and only if it is a smooth solution to \eqref{700}. Since $\operatorname{Dom}(\phi)\subseteq \operatorname{Dom}(\psi)$, it is clear that a solution to \eqref{700} is also a solution to \eqref{2}. Conversely, we need to show that a smooth solution to \eqref{2} satisfies $\frac{1}{2}(w')^2\in\operatorname{Dom}(\phi)$. If $\operatorname{Dom}(\phi) = \operatorname{Dom}(\psi)$ there is nothing to prove, so we suppose henceforth that $\operatorname{Dom}(\phi)\subsetneqq \operatorname{Dom}(\psi)$. Then by Lemma \ref{140}, $\operatorname{Dom}(\psi) \cap (0,\infty) = (0,\infty) = \operatorname{Dom}(\phi)$, and hence it suffices to show that $w'>0$. Since $w'(0) = 1$ (by Proposition \ref{A'} and Lemma \ref{p0}), it follows that $w'(t)>0$ for $t$ sufficiently small. Suppose for a contradiction that $w$ has a critical point at some $t_0$, i.e.~$w'(t_0)=0$. Then by the equation \eqref{2}, we have
	\begin{align*}
	w(t_0) w''(t_0) = - \psi(0) < -\frac{1}{2},
	\end{align*}
	where we have used $f(\psi(0),0,\dots,0) =f(\frac{1}{2}, \dots,\frac{1}{2}) = \frac{1}{2}$ and monotonicty of $f$ to assert that $\psi(0) > \frac{1}{2}$. In particular, since $w(t_0)>0$, it follows that $w''(t_0)<0$ and hence $t_0$ is a strict local maximum for $w$. In particular, we have shown that any critical point of $w$, if one exists, is a strict local maximum. A simple argument then shows that $t_0$ is the unique critical point of $w$ and $w$ is decreasing for $t>t_0$. Therefore $w(t) < t$ for sufficiently large $t$, contradicting the lower bound $w(t) \geq t$ from Proposition \ref{A'}. 
\end{proof}

Up until now, we have not needed any concavity assumption on $f$. In fact, in the proof of Theorems \ref{B}, \ref{A} and \ref{D} we will not need the full strength of the concavity assumption in \eqref{26'}, leading us to introduce the following weaker condition:
\begin{align}
f(t,1,\dots,1) \leq f(e) + \frac{\partial f}{\partial \lambda_1}(e)(t-1) \text{ for all }t>-\mu_\Gamma^+. \tag{\ref{26'}$'$} \label{23'''}
\end{align}
The properties of $\phi$ in the following lemma will be used in subsequent arguments: 

\begin{lem}\label{115}
	Suppose that $(f,\Gamma)$ satisfies \eqref{21'}--\eqref{25'} and \eqref{23'''}. Then for $\phi$ defined as above,
		\begin{align}\label{119}
		A(s) \defeq \frac{1}{2ns(s-\frac{1}{2})}\bigg(\frac{2n(s-\frac{1}{2})}{1-\phi(s)}-1\bigg) \geq \frac{1}{ns}>0 \quad \text{for }s>\frac{1}{2}.
		\end{align}

\end{lem}

\begin{proof}
	We first claim that 
	\begin{align}\label{114}
	-\phi'\bigg(\frac{1}{2}\bigg)= 2\bigg(\frac{\partial f}{\partial \lambda_1}(e)\bigg)^{-1} = 2n.
	\end{align}
	The first identity in \eqref{114} follows from differentiating the equation $f(\phi(s),1,\dots,1) = \frac{1}{2s}$ with respect to $s$ and evaluating at $s=\frac{1}{2}$. To obtain the second identity in \eqref{114}, first observe by 1-homogeneity of $f$ that $f(\mu) = \sum_i\frac{\partial f}{\partial \lambda_i}(\mu)\mu_i$ for all $\mu\in\Gamma$, and thus
	\begin{align*}
	\sum_{i=1}^n \frac{\partial f}{\partial\lambda_i}(\mu) = f(\mu) +  \sum_{i=1}^n \frac{\partial f}{\partial \lambda_i}(\mu)(1-\mu_i). 
	\end{align*}
	Taking $\mu = e$ and appealing to symmetry of $f$, we therefore see that
	\begin{align*}
	n\frac{\partial f}{\partial \lambda_1}(e) = \sum_{i=1}^n \frac{\partial f}{\partial\lambda_i}(e) = f(e) = 1.
	\end{align*}
	This proves the second identity in \eqref{114}.  \medskip 
	
	We now prove \eqref{119}. First observe by \eqref{23'''} and \eqref{114} that
	\begin{align}\label{109}
	\frac{1}{2s} = f(\phi(s),1,\dots,1) \leq f(e)+ \frac{\partial f}{\partial \lambda_1}(e)(\phi(s) -1) = 1 - \frac{1}{n}(1-\phi(s)). 
	\end{align}
	Rearranging \eqref{109} and using the fact that $\phi(s)<1$ for $s>\frac{1}{2}$, it follows that 
	\begin{align}\label{4}
	\frac{s-\frac{1}{2}}{1-\phi(s)} \geq  \frac{s}{n},
	\end{align}
	and substituting \eqref{4} back into the definition of $A$, we obtain
	\begin{align*}
	A(s)  \geq \frac{1}{2ns(s-\frac{1}{2})}(2s-1) = \frac{1}{ns},
	\end{align*}
	as claimed. 
\end{proof}

\subsection{Existence, uniqueness and asymptotics of solutions with positive initial data}

We begin our study of \eqref{700} by considering solutions with positive initial data.  Solutions satisfying $w(0) = 0$ (for which ellipticity degenerates as $t\rightarrow 0$) will then be obtained through a limiting process. We start with the following existence and uniqueness result:

\begin{prop}\label{p1}
	Suppose that $(f,\Gamma)$ satisfies \eqref{21'}--\eqref{25'}, and let $\phi$ be as defined as above. Let $\delta>0$ and $p \geq 1$, and fix an initial starting time $\tau \geq 0$. Then the initial value problem 
	\begin{align}\label{118}
	\begin{cases}
	ww'' = \frac{1}{2}(w')^2\big[1 - \phi\big(\frac{1}{2}(w')^2\big)\big], \,\,\, \frac{1}{2}(w')^2 \in \operatorname{Dom}(\phi)  & \text{on }(\tau,\tau+T)\\
	w(\tau) = \delta, \,\, w'(\tau) = p
	\end{cases}
	\end{align} 
	admits a unique smooth solution $w=w_{\delta,p,\tau}$ for some maximal value $0<T\leq \infty$. If $p=1$, then $T = \infty$ and the solution is given by $w(t) = t - \tau + \delta$. If $p>1$, then $w'>p$ and $w''>0$ in $(\tau,\tau+T)$. Finally, if $T<\infty$, then $w'(t)\rightarrow\infty$ as $t\rightarrow T^-$. 
\end{prop}

\begin{proof}
	First observe that 
	\begin{align*}
	f(x,y) \defeq \frac{y^2}{2x}\bigg(1-\phi\Big(\frac{1}{2}y^2\Big)\bigg)
	\end{align*}
	is smooth (and hence locally Lipschitz) on $\{(x,y)\in [\delta ,\infty)\times[\frac{1}{2},\infty)\}$, since $\delta>0$ and $[\frac{1}{2},\infty)\subset\operatorname{Dom}(\phi)$. It follows from standard ODE theory that there exists a unique smooth short time solution to $ww'' = \frac{1}{2}(w')^2\big[1 - \phi\big(\frac{1}{2}(w')^2\big)\big]$ satisfying $w(\tau) = \delta$ and $w'(\tau) = p\geq 1$. If $p=1$, is easy to check that $w(t) = t-\tau + \delta$ is the unique solution for short time, and hence for all $t$ by applying uniqueness ad infinitum. If $p>1$, then since $\phi$ is decreasing and $\phi(\frac{1}{2}p^2)<1$, a simple argument by contradiction using the equation in \eqref{118} yields $w''>0$ and hence $w'>p$ for as long as solutions exist. In particular, $\frac{1}{2}(w')^2\in \operatorname{Dom}(\phi)$. The final assertion follows from the fact that $w \geq \delta$ and $f$ is (globally) Lipschitz on $[\delta,\infty)\times[\frac{1}{2},Y]$ for any fixed $Y<\infty$. 
\end{proof}

In the next two results we find a first integral for \eqref{118} and derive some asymptotic properties involving the resulting Hamiltonian. Recall the function $A:(\frac{1}{2},\infty)\rightarrow (0,\infty)$ defined in Lemma \ref{115}, and note that since $\phi$ is smooth near $s=\frac{1}{2}$ with $\phi(\frac{1}{2}) = 1$ and $\phi'(\frac{1}{2})= -2n$, $A$ extends to a continuous function on $[\frac{1}{2},\infty)$. We may therefore define
\begin{align}\label{900}
B(s) = \int_{1/2}^s A(\zeta)\,d\zeta \quad \text{and} \quad G(s) = \bigg(\frac{s-\frac{1}{2}}{s}\bigg)^{\frac{1}{n}}e^{B(s)}.
\end{align}
Note that since $A(s) \geq \frac{1}{ns}>0$ (by Lemma \ref{115}), $G$ is a smooth strictly increasing function from $[\frac{1}{2},\infty)\rightarrow [0,\infty)$. It is therefore invertible with a smooth strictly increasing inverse 
\begin{align}\label{901}
K\defeq G^{-1}:[0,\infty)\rightarrow\bigg[\frac{1}{2},\infty\bigg).
\end{align}

\begin{prop}\label{p2}
	Suppose that $(f,\Gamma)$ satisfies \eqref{21'}--\eqref{25'} and \eqref{23'''}, and let $G$ and $K$ be defined as in \eqref{900} and \eqref{901}. Fix $\delta>0$ and $p\geq 1$, and let $w=w_{\delta,p}$ be as in Proposition \ref{p1} with $\tau=0$. Then
	\begin{align}\label{8}
	w' = \sqrt{2K(bw)},
	\end{align}
	where
	\begin{align}\label{110}
	b  \defeq  \frac{1}{\delta} G\bigg(\frac{p^2}{2}\bigg)= \frac{1}{\delta}\bigg(\frac{p^2-1}{p^2}\bigg)^{\frac{1}{n}}e^{B\big(\frac{p^2}{2}\big)}.
	\end{align}
	\end{prop}

	\begin{proof}
	If $p=1$ then $w(t) = t+\delta$ and hence the result is clear. So suppose $p>1$ and denote $z = \frac{1}{2}(w')^2$. Note that $z$ is strictly increasing in $t$ by Proposition \ref{p1}, so in particular $z \geq \frac{p^2}{2}>\frac{1}{2}$. Since $K\defeq G^{-1}$, \eqref{8} will follow once we have shown that
	\begin{align}\label{6}
	G(z) = \bigg(\frac{z-\frac{1}{2}}{z}\bigg)^{\frac{1}{n}} e^{B(z)} = bw.
	\end{align}
	We start by noting $\frac{dz}{dw} = \frac{dz}{dt}\frac{dt}{dw} = w''$ and hence the equation in \eqref{700} reads
	\begin{align}\label{902}
	\frac{dz}{z(1-\phi(z))} = \frac{dw}{w}.
	\end{align}
	Recalling the definition of $A$, we rewrite \eqref{902} as
	\begin{align}\label{5}
	\bigg(\frac{1}{2nz(z-\frac{1}{2})} + A(z)\bigg)\,dz = \frac{dw}{w}.
	\end{align}
	Integrating in \eqref{5}, recalling the definition of $B$ and using the initial conditions for $w$, we obtain
	\begin{align}\label{131}
	\frac{1}{n}\ln\bigg(\frac{z-\frac{1}{2}}{z}\bigg) - \frac{1}{n} \ln\bigg(\frac{p^2-1}{p^2}\bigg) + B(z) - B\bigg(\frac{p^2}{2}\bigg) = \ln\bigg(\frac{w}{\delta}\bigg).
	\end{align}
	Exponentiating both sides of \eqref{131} yields \eqref{6}, and therefore \eqref{8}.\end{proof}

	\begin{lem}\label{904}
		Suppose that $(f,\Gamma)$ satisfies \eqref{21'}--\eqref{25'} and \eqref{23'''}. Then the function $K$ defined in \eqref{901} satisfies the following asymptotics:
		\begin{align}\label{111}
		x^{1+\mu_\Gamma^+ + o(1)} = K(x) \leq  e^{O(1)}x^{1+\mu_\Gamma^+}  \quad \text{as }x\rightarrow+\infty.
		\end{align}
	\end{lem}
	\begin{proof}
	First observe
	 \begin{align*}
	B(s) = \int_{1/2}^s A(\zeta)d\zeta = \int_{1/2}^s \bigg[\frac{1}{\zeta(1-\phi(\zeta))} - \frac{1}{n}\bigg(\frac{1}{\zeta - \frac{1}{2}} - \frac{1}{\zeta}\bigg)\bigg]\,d\zeta. 
	\end{align*}
	Then since $\phi(\zeta)\rightarrow-\mu_\Gamma^+$ as $\zeta\rightarrow\infty$ and $\phi(\zeta) > -\mu_\Gamma^+$ for all $\zeta$, we see that
	\begin{align*}
	\frac{1}{1+\mu_\Gamma^+} \ln s + O(1) \leq B(s) = \frac{1+o(1)}{1+\mu_\Gamma^+}\ln s + O(1) = \frac{1+o(1)}{1+\mu_\Gamma^+}\ln s 
	\end{align*} 
	as $s\rightarrow\infty$. Therefore
	\begin{align}\label{112'}
	G(s) = \bigg(\frac{s-\frac{1}{2}}{s}\bigg)^{\frac{1}{n}}e^{B(s)}  = e^{O(1)}e^{B(s)} \geq e^{O(1)}s^{\frac{1}{1+\mu_\Gamma^+}} \quad \quad \text{as }s\rightarrow+\infty
	\end{align}
	and
	\begin{align}\label{112}
	G(s) = \bigg(\frac{s-\frac{1}{2}}{s}\bigg)^{\frac{1}{n}} e^{B(s)} = s^{O(\frac{1}{s\ln s})}e^{B(s)} = s^{O(\frac{1}{s\ln s})}s^{\frac{1+o(1)}{1+\mu_\Gamma^+}}= s^{\frac{1+o(1)}{1+\mu_\Gamma^+}}
	\end{align}
	as $s\rightarrow\infty$. 
	Since $K=G^{-1}$, the desired bound \eqref{111} then follows from \eqref{112'} and \eqref{112}. 
\end{proof}

Next we address the maximal existence time of solutions to \eqref{118}:

\begin{lem}\label{p3}
	Fix $\delta>0$, $p>1$ and $\tau \geq 0$, and let $[\tau,\tau+T)$ be the maximal time interval on which the smooth solution $w = w_{\delta,p,\tau}$ to \eqref{118} exists. If $\mu_\Gamma^+ \leq 1$ then $T=\infty$, and if $\mu_\Gamma^+ >1$ then $T<\infty$. 
\end{lem}

\begin{proof}
	\noindent\textbf{Case 1:} $\mu_\Gamma^+ \leq 1$. By Proposition \ref{p1} it suffices to show that $w'$ does not blow up in finite time. In the proof we will use the fact that $w'>0$ (again by Proposition \ref{p1}) without explicit reference. To this end, observe by \eqref{118} and the lower bound $\phi(s) \geq -\mu_\Gamma^+$ that
	\begin{align*}
	\frac{w''}{w'} \leq \frac{(1+\mu_\Gamma^+)w'}{2w}.
	\end{align*}
Integrating and using the initial conditions for $w$ then yields
	\begin{align*}
	\ln\bigg(\frac{w'}{p}\bigg) \leq \frac{1+\mu_\Gamma^+}{2}\ln\bigg(\frac{w}{\delta}\bigg),
	\end{align*}
	and therefore
	\begin{align*}
	w' \leq p\bigg(\frac{w}{\delta}\bigg)^{\frac{1+\mu_\Gamma^+}{2}}. 
	\end{align*}
	Since $w \geq \delta$ and we assume $\frac{1+\mu_\Gamma^+}{2} \leq 1$, it follows that for some constant $C$ depending on $p,\delta$ and $\mu_\Gamma^+$ that
	\begin{align}\label{3}
	w' \leq Cw. 
	\end{align}
	Therefore $\delta \leq w(t) \leq \delta e^{Ct}$, and substituting this back into \eqref{3} yields $w'(t) \leq C\delta e^{Ct}$, proving that $w'$ does not blow up in in finite time. \medskip 
	
	\noindent\textbf{Case 2:} $\mu_\Gamma^+ > 1$. Recall $w' = \sqrt{2K(bw)}$. By Lemma \ref{904}, $\sqrt{2K(bw)} =  \sqrt{2}(bw)^{\frac{1+\mu_\Gamma^+ + o(1)}{2}}$ as $w\rightarrow\infty$. Since $\mu_\Gamma^+>1$, it follows that there exist constants $C>0$ and $\ep>0$ such that $\sqrt{2K(bw)} \geq Cw^{1+\ep}$. Therefore
	\begin{align*}
	T = \int_\delta^{w_{\delta,p,\tau}(\tau+T)}  \frac{ds}{\sqrt{2K(bs)}} \leq \int_\delta^\infty \frac{ds}{\sqrt{2K(bs)}} \leq \frac{1}{C}\int_{\delta}^\infty s^{-1-\ep}\,ds<\infty. 
	\end{align*}

\end{proof}

\subsection{Proof of Theorem \ref{B}}

Recall from \eqref{110} that $b  = \frac{1}{\delta} G(\frac{p^2}{2})$, from which it is clear that $p$ is uniquely determined by $b\delta$. Before giving the proof of Theorem \ref{B} we need one final key lemma:

\begin{lem}\label{702}
	Suppose $\mu_\Gamma^+ \leq 1$ and fix $a>1$. Then for each $\delta\in(0,a-1)$, there exists a unique value $p = p^{(a)}(\delta)$ (equivalently, a unique value $b = b^{(a)}(\delta)$) such that $w_{\delta,p}(1) = a$. Moreover, $b^{(a)}$ extends to a strictly decreasing continuous bijection from $[0,a-1]$ to $[0,b^{(a)}(0)]$. In particular, there exists $C = C(a)>0$ such that $0< C(a)^{-1} \leq b^{(a)}(\delta) \leq C(a)$ for $0\leq \delta<\frac{a-1}{2}$.
\end{lem}
\begin{proof}
	Throughout the proof $a$ is a fixed constant and so we suppress dependence on $a$ in our notation. Recall from Lemma \ref{p3} that since $\mu_\Gamma^+ \leq 1$, the solution $w=w_{\delta,p}$ (with $\tau=0$) exists for infinite time for each $\delta>0$ and $p>1$. Then $w_{\delta,p}(1) = a$ if and only if 
	\begin{align}\label{7}
	1 = \int_\delta^{a} \frac{ds}{\sqrt{2K(bs)}} \eqdef \gamma(b,\delta). 
	\end{align}
	But \eqref{6} implies $K(0) = \frac{1}{2}$ and $\lim_{b\rightarrow\infty} K(bs) = +\infty$ for any $s>0$, from which it follows that $\gamma(0,\delta) = a - \delta> 1$ and $\lim_{b\rightarrow+\infty}\gamma(b,\delta) = 0$. Therefore, by the intermediate value theorem and strict monotonicity of $K$, there is a unique solution $b=b(\delta)$ to \eqref{7}.\medskip 
	
	Noting that $\frac{\partial\gamma}{\partial\delta}<0$ and $\frac{\partial\gamma}{\partial b}<0$, it follows from the implicit function theorem that $b$ is differentiable with $b'<0$ in $(0,a-1)$. It follows from nonnegativity and monotonicity of $b$ that $b$ extends to a strictly decreasing nonnegative continuous function on $(0,a-1]$. Since $\gamma(0,a-1) = 1$, it follows that $b(a-1) = 0$. To conclude the proof, it remains to show that $\lim_{\delta\rightarrow 0} b(\delta)<\infty$. Recall from Lemma \ref{904} that $K(x) = x^{1+\mu_\Gamma^+ + o(1)}$ as $x\rightarrow\infty$, and hence there exist constants $C>0$ and $0<\ep<1+\mu_\Gamma^+$ such that $K(x) \geq \frac{1}{C}(1+x)^{1+\mu_\Gamma^+ - \ep}$ for all $x \geq 0$. Therefore, for a constant $C>0$ which may change between inequalities but remains independent of $b$, we have by \eqref{7}
	\begin{align}\label{906}
	1 \leq C\int_\delta^a (1+bs)^{-\frac{1+\mu_\Gamma^+-\ep}{2}}\,ds = \frac{C}{b}\int_{1+b\delta}^{1+ba} t^{-\frac{1+\mu_\Gamma^+-\ep}{2}}\,dt \leq \frac{C}{b}(1+ba)^{\frac{1-\mu_\Gamma^+ +\ep}{2}}. 
	\end{align}
	Since $\frac{1-\mu_\Gamma^+ +\ep}{2}<1$, it follows from \eqref{906} that $b=b(\delta)$ is bounded from above independently of $\delta$. In particular, $b$ extends to a strictly decreasing continuous function on $[0,a-1]$. 
\end{proof}

\begin{proof}[Proof of Theorem \ref{B}]
For $a>1$ and $0<\delta<a-1$, let us now denote by $w_\delta^{(a)}$ the global solution to \eqref{700} satisfying $w_\delta^{(a)}(0) = \delta$ and $w_\delta^{(a)}(1)=a$, which we know exists and is unique by Proposition \ref{p1} and Lemmas \ref{p3} and \ref{702}.\medskip 

\noindent\textbf{Step 1:} In this step we show that we can take $\delta\rightarrow 0$ to obtain a global solution $w^{(a)}$ satisfying $w^{(a)}(0)=0$ and $w^{(a)}(1) = a$. It suffices to show that for each $k\in\mathbb{N}\cup\{0\}$ and $t>0$, there exists a constant $C=C(a,k,t)$ which is independent of $\delta$ such that
	\begin{align}\label{134}
	\|w_\delta^{(a)}\|_{C^k([0,t])} \leq C(a,k,t).
	\end{align}
	Indeed, once \eqref{134} is established we may take $\delta\rightarrow 0$ along a suitable sequence to obtain a smooth solution to \eqref{8} on $[0,\infty)$ with $w^{(a)}(0)=0$, $w^{(a)}(1) = a$ and $b = b^{(a)}(0)$.\medskip
	
	Once the estimate \eqref{134} is established for $k=0$, the estimate for $k\geq 1$ is then routine and proceeds as follows. We have the lower bound $(w^{(a)}_\delta)'\geq 0$ since $w^{(a)}_\delta$ is increasing, whereas an upper bound for $(w_\delta^{(a)})'$ follows from the upper bound we have just established for $w_\delta^{(a)}$, the upper bound $b=b^{(a)}(\delta) \leq C(a)$ from Lemma \ref{702}, and the equation \eqref{8}. Higher order estimates then follow from differentiating the equation \eqref{8} and again using the uniform bounds on $b$ from Lemma \ref{702}. \medskip
	
	For the case $k=0$, since $w^{(a)}_\delta$ is increasing, it suffices to show that for each fixed $t>0$, $w^{(a)}_\delta(t)$ is bounded as $\delta\rightarrow 0$. To this end, first observe that for each $t>0$,
	\begin{align*}
	t = \int_\delta^{w_{\delta}^{(a)}(t)} \frac{ds}{\sqrt{2K(bs)}} & = \frac{1}{b}\int_{b\delta}^{bw^{(a)}_\delta(t)} \frac{ds}{\sqrt{2K(s)}} \nonumber \\
	& = \frac{1}{b}\int_{0}^{bw^{(a)}_\delta(t)} \frac{ds}{\sqrt{2K(s)}} - \frac{1}{b}\int_{0}^{b\delta} \frac{ds}{\sqrt{2K(s)}} \nonumber \\
	&  \geq \frac{1}{C(a)}\int_{0}^{C(a)^{-1}w^{(a)}_\delta(t)}\frac{ds}{\sqrt{2K(s)}} - C(a)\int_0^{C(a)}\frac{ds}{\sqrt{2K(s)}}
	\end{align*}
	where $C(a)$ is the constant in Lemma \ref{702}. Therefore the integral 
	\begin{align}\label{905}
	\int_{0}^{C(a)^{-1}w^{(a)}_\delta(t)}\frac{ds}{\sqrt{2K(s)}}
	\end{align}
	is bounded as $\delta\rightarrow 0$. On the other hand, recall the upper bound $K(x) \leq e^{O(1)}x^{1+\mu_\Gamma^+}$ as $x\rightarrow\infty$ in Lemma \ref{904}, from which it follows that there exists a constant $D>0$ such that $K(x) \leq D(1+x^{1+\mu_\Gamma^+})$ for all $x\geq 0$. Since $\mu_\Gamma^+ \leq 1$, we therefore have
	\begin{align}\label{133}
	\int_{0}^\infty \frac{ds}{\sqrt{2K(s)}} \geq \frac{1}{\sqrt{2D}}\int_0^\infty (1+s^{1+\mu_\Gamma^+})^{-\frac{1}{2}}\,ds = \infty.
	\end{align}
	It follows from \eqref{133} and boundedness of the integral in \eqref{905} as $\delta\rightarrow 0$ that for each fixed $t>0$, $w_{\delta}^{(a)}(t)$ is bounded as $\delta\rightarrow 0$. \medskip 
	
	Before moving to Step 2, we make the following remark: 
	
	\begin{rmk}
		Although not needed in the proof, we point out that the limit $w^{(a)}$ obtained in Step 1 does not depend on the sequence along which $\delta$ is taken to zero, which allows us to talk of unambiguously of `the' solution $w^{(a)}$ obtained in Step 1. Indeed, the comparison principle implies that $w_{\delta_1}^{(a)}\leq w_{\delta_2}^{(a)}$ on $[0,1]$ when $\delta_1 < \delta_2$, and thus $w^{(a)}$ is uniquely determined on $[0,1]$, as it is the limit of a decreasing sequence of functions. But by the uniqueness part of Proposition \ref{p1}, $w^{(a)}$ must agree with $w_{\delta, p, \frac{1}{2}}$ on $\{x_n \geq \frac{1}{2}\}$, where $\delta = w^{(a)}(\frac{1}{2})>0$ and $p = (w^{(a)})'(\frac{1}{2}) \geq 1$. Therefore the limit $w^{(a)}$ is uniquely determined for $x_n \geq 0$.
	\end{rmk}

	\noindent\textbf{Step 2:} For $a>0$, let $w^{(a)}$ denote the solution obtained in Step 1 and define $w^{(0)}(x_n) \defeq x_n$. In this step we show that the solutions $\{w^{(a)}\}_{a \geq 0}$ satisfy the six properties stated in the theorem. The first property is true by definition. The second property is trivial when $a=0$ and is true for $a>0$ by the construction of $w^{(a)}$ in Step 1. In light of Proposition \ref{p1}, the lower bounds $(w^{(a)})' \geq 1$ and $(w^{(a)})'' \geq 0$ also follow from the construction in Step 1 for $a>0$, and are trivial for $a=0$. The fact that $(w^{(a)})'(0)=1$ for each $a>0$ follows from Proposition \ref{A'} and Lemma \ref{p0'}, and is again trivial for $a=0$. This proves the third and fourth properties.  \medskip 
	
	To prove the fifth property, let $a<b$ and suppose for a contradiction that there exists $t_0>0$ such that $w^{(a)}(t_0) = w^{(b)}(t_0)$. By translation invariance of the equation, for each $\ep>0$, $w^{(a)}_\ep(t) \defeq w^{(a)}(t+\ep)$ is a solution on $(0,\infty)$ satisfying $w^{(a)}_\ep(0) = w^{(a)}(\ep) > 0 = w^{(b)}(0)$ and $w^{(a)}_\ep(t_0) = w^{(a)}(t_0+\ep) > w^{(a)}(t_0) = w^{(b)}(t_0)$, with the latter inequality following from the fact that $w^{(a)}$ is increasing (Property 3). By the comparison principle, it follows that $w^{(a)}_\ep \geq w^{(b)}$ on $[0,t_0]$, and taking $\ep\rightarrow 0$ we see that $w^{(a)} \geq w^{(b)}$ on $[0,t_0]$. Likewise, $w^{(a)}_{-\ep}(t) \defeq w^{(a)}(t-\ep)$ is a solution on $(\ep,\infty)$ satisfying $w^{(a)}_{-\ep}(\ep) = 0 < w^{(b)}(\ep)$ and $w^{(a)}_{-\ep}(t_0) = w^{(a)}(t_0-\ep)< w^{(a)}(t_0) = w^{(b)}(t_0)$, and so the comparison principle implies $w^{(a)}_{-\ep} \leq w^{(b)}$ on $[\ep,t_0]$. Taking $\ep\rightarrow 0$, we see that $w^{(a)} \leq w^{(b)}$ on $[0,t_0]$. Therefore $w^{(a)} = w^{(b)}$ on $[0,t_0]$. By the uniqueness part of Proposition \ref{p1}, we therefore have $w^{(a)} = w^{(b)}$ on $[\frac{t_0}{2},\infty)$, and hence on $[0,\infty)$. This contradicts $w^{(a)}(1) = a < b = w^{(b)}(1)$. Therefore no such $t_0$ exists, and hence either $w^{(a)}<w^{(b)}$ or $w^{(b)}<w^{(a)}$ on $(0,\infty)$. Since $w^{(a)}(1)<w^{(b)}(1)$, clearly it is the former inequality that holds. \medskip 
	
	It remains to prove the sixth property. Local completeness near $\partial\mathbb{R}_+^n$ follows from Proposition \ref{A'} and Lemma \ref{p0'} for $a>0$, and global completeness is well-known (and routine to show) for $a=0$, since $w^{(0)}$ is the hyperbolic solution. To show that $g_{w^{(a)}}$ is globally incomplete for $a>0$, it suffices to show that the curve $\gamma:(1,\infty)\rightarrow \mathbb{R}_+^n$ given by $\gamma(t) = (0',t)\in\mathbb{R}^{n-1}\times[0,\infty)$ has finite length with respect to $g_{w^{(a)}}$, that is
	\begin{align}\label{703}
	\int_1^\infty (w^{(a)}(t))^{-1}\,dt<\infty. 
	\end{align}
	To this end, first recall that that $(w^{(a)})' = \sqrt{2K(b^{(a)}w^{(a)})}$. By the asymptotics for $K$ in Lemma \ref{904}, it follows that there exist constants $t_0>0$ and $C>0$ depending on $a$ such that $(w^{(a)})'(t)   \geq C(w^{(a)})^{\frac{1}{4}}$ for $t \geq t_0$. Therefore $w^{(a)}(t) \geq Ct^{\frac{4}{3}}$ for sufficiently large $t$, from which \eqref{703} follows easily. 
	\end{proof}

\subsection{Proof of Theorem \ref{D}}

	\begin{proof}[Proof of Theorem \ref{D}]
	We first show that for any $0<\ep<\mathcal{E}$, $\inf_{[\ep,\mathcal{E}]}w^{(a)}\rightarrow \infty$ as $a\rightarrow\infty$. Since $w^{(a)}$ is increasing, it suffices to show that $w^{(a)}(\ep) \rightarrow \infty$ as $a\rightarrow\infty$.  To this end, first observe that
	\begin{align*}
	\ep= \int_0^{w^{(a)}(\ep)} \frac{ds}{\sqrt{2K(b^{(a)}s)}} = \frac{1}{b^{(a)}}\int_0^{b^{(a)}w^{(a)}(\ep)} \frac{ds}{\sqrt{2K(s)}}.
	\end{align*}
On the other hand, since $w^{(a)}(1) = a$, we have
\begin{align*}
1 = \int_0^a  \frac{ds}{\sqrt{2K(b^{(a)}s)}} = \frac{1}{b^{(a)}}\int_0^{b^{(a)}a} \frac{ds}{\sqrt{2K(s)}},
\end{align*}
which implies (in light of \eqref{133}) that $b^{(a)}\rightarrow\infty$ as $a\rightarrow\infty$. \medskip 

Now, by Lemma \ref{904} and the fact that $\mu_\Gamma^+ \geq 0$, we have $K(x) \geq x^{1/2}$ for $x$ sufficiently large and hence $K(x) \geq Cx^{1/2}$ for all $x \geq 0$ for some constant $C>0$. Therefore
	\begin{align}\label{122}
	\ep  = \frac{1}{b^{(a)}}\int_0^{b^{(a)} w^{(a)}(\ep)} \frac{ds}{\sqrt{2K(s)}} & \leq  \frac{1}{\sqrt{2C}b^{(a)}}\int_0^{b^{(a)} w^{(a)}(\ep)} s^{-\frac{1}{4}}\,ds  \nonumber \\
	& = \frac{4}{3\sqrt{2C}(b^{(a)})^\frac{1}{4}} (w^{(a)}(\ep))^{\frac{3}{4}}.
	\end{align}
	It is then clear from \eqref{122} and the fact that $b^{(a)}\rightarrow\infty$ as $a\rightarrow \infty$ that $w^{(a)}(\ep)\rightarrow\infty$. Since $w^{(a)}(0) = 0$, it follows that $\sup_{t\in[0,\ep]} (w^{(a)})'(t) \rightarrow\infty$ as $a\rightarrow\infty$ and hence, by convexity of each $w^{(a)}$, $\inf_{t\in[\ep,\mathcal{E}]} (w^{(a)})'(t)\rightarrow\infty$ as $a\rightarrow\infty$. 
\end{proof}

\subsection{Proof of Theorem \ref{A}}\label{s5}

\begin{proof}[Proof of Theorem \ref{A}]
		Suppose that $w$ is a continuous viscosity solution to \eqref{15}. By Proposition \ref{A'}, we know that $w = w(x_n)$ with $w(0)=0$, $w'(0)=1$ and $w(x_n) \geq x_n$. By Lemma \ref{p0'}, we also know that $w$ is a smooth solution to \eqref{700}. We split the rest of the proof into two cases, according to whether $\mu_\Gamma^+>1$ or $\mu_\Gamma^+ \leq 1$. \medskip 
	
	\noindent\textbf{Case 1:} $\mu_\Gamma^+>1$.  We wish to show that $w(x_n)=x_n$. It suffices to show that $w' \leq 1$ for $x_n \geq 0$, since this implies $w(x_n) = \int_0^{x_n}w'(t)\,dt \leq x_n$, which together with $w(x_n) \geq x_n$ further implies $w(x_n) = x_n$. To this end, suppose for a contradiction that there exists $t_1 \geq 0$ such that $w'(t_1) >1$. Since $w'(0)=1$, it must be the case that $t_1>0$. Then necessarily $w(t_1) > t_1$. By the uniqueness part of Proposition \ref{p1}, $w$ coincides with $w_{w(t_1), w'(t_1), t_1}$ on the half-space $\{x_n \geq t_1\}$. But since $\mu_\Gamma^+>1$, Lemma \ref{p3} implies that $w_{w(t_1), w'(t_1), t_1}$ does not exist for all $x_n \geq t_1$. Therefore $w$ does not exist for all $x_n \geq t_1$, a contradiction. \medskip 
	
	\noindent\textbf{Case 2:} $\mu_\Gamma^+ \leq 1$. Letting $a= w(1)$, we will show that $w=w^{(a)}$, where $w^{(a)}$ is the solution obtained in Theorem \ref{B}. We use the same argument as in the proof of the fifth property in the proof of Theorem \ref{B}. By translation invariance of the equation, for each $\ep>0$, $w^{(a)}_\ep(t) \defeq w^{(a)}(t+\ep)$ is a solution on $(0,\infty)$ satisfying $w^{(a)}_{\ep}(0) = w^{(a)}(\ep) > 0 = w(0)$ and $w^{(a)}_\ep(1) = w^{(a)}(1+\ep) > w^{(a)}(1) = a = w(1)$, since $w^{(a)}$ is increasing. By the comparison principle, it follows that $w_\ep^{(a)} \geq w$ on $[0,1]$, and taking $\ep\rightarrow 0$ we see that $w^{(a)} \geq w$ on $[0,1]$. Likewise, $w_{-\ep}^{(a)} \defeq w^{(a)}(t-\ep)$ is a solution on $(\ep,\infty)$ satisfying $w_{-\ep}^{(a)}(\ep) = 0 < w(\ep)$ and $w_{-\ep}^{(a)}(1) < a = w(1)$, and so the comparison principle implies $w_{-\ep}^{(a)} \leq w$ on $[\ep,1]$. Taking $\ep\rightarrow 0$, we see that $w^{(a)} \leq w$ on $[0,1]$. Therefore $w^{(a)} = w$ on $[0,1]$. By the uniqueness part of Proposition \ref{p1}, we therefore have $w = w^{(a)}$ on $[\frac{1}{2},\infty)$, and hence on $[0,\infty)$. 
\end{proof}

\begin{appendices}

	\addtocontents{toc}{\protect\setcounter{tocdepth}{-1}}
	
\section{Proof of Lemmas \ref{140} and \ref{600}}\label{appa}

In this appendix we give the proof of Lemmas \ref{140} and \ref{600}. 

\begin{proof}[Proof of Lemma \ref{140}]
	First suppose that $\eta<\infty$. We claim that $\operatorname{Dom}(\psi) = (\frac{1}{2\eta},\infty)$. By \eqref{139} and strict monotonicity of $f$, we know that $f(R,1,\dots,1)<\eta$ for any $R<\infty$. It follows that $f(s_1,\frac{1}{2\eta},\dots,\frac{1}{2\eta})<\frac{1}{2}$ for all $s_1$, which implies $\frac{1}{2\eta}\not\in\operatorname{Dom}(\psi)$. Since $\operatorname{Dom}(\psi)$ is connected and contains $\operatorname{Dom}(\phi) = (\frac{1}{2\eta},\infty)$, the claim follows.\medskip 
	
	Now suppose that $\eta = \infty$, so that $(0,\infty)=\operatorname{Dom}(\phi)\subseteq \operatorname{Dom}(\psi)$. If $(0,\infty) = \operatorname{Dom}(\psi)$, then the first assertion of the lemma holds. We therefore wish to show that if $(0,\infty)\not= \operatorname{Dom}(\psi)$, then $\operatorname{Dom}(\psi)=\mathbb{R}$. To this end, note by connectedness and openness of $\operatorname{Dom}(\psi)$,  $(-\delta,\infty)\subseteq \operatorname{Dom}(\psi)$ for some $\delta>0$. Therefore, one can uniquely define as in \cite{LN14b} a quantity $\mu_\Gamma^->0$ by 
	\begin{align*}
	(\mu_\Gamma^-, -1, \dots, -1)\in\partial\Gamma. 
	\end{align*}
	Defining $\eta_- \defeq \lim_{R\rightarrow+\infty}f(R, -1, \dots, -1) > 0$, it follows that for each $s\in (\frac{1}{2\eta_-}, \infty)$ there exists a unique number $\phi_-(s)\in(\mu_\Gamma^-,\infty)$ such that $f(\phi_-(s), -1, \dots, -1) = \frac{1}{2s}$, or equivalently $f(s\phi_-(s), -s, \dots, -s) = \frac{1}{2}$. Thus $(-\infty, -\frac{1}{2\eta_-})\subset \operatorname{Dom}(\psi)$. By connectedness of $\operatorname{Dom}(\psi)$ and the fact $(0,\infty)\subseteq \operatorname{Dom}(\psi)$, it follows that $\operatorname{Dom}(\psi) = \mathbb{R}$, as required. \medskip 
	
	It remains to show that if $(1,0,\dots,0)\in\Gamma$ then $\eta=\infty$ and $\operatorname{Dom}(\psi)=\mathbb{R}$. First note that if $(1,0,\dots,0)\in\Gamma$, then $0\in\operatorname{Dom}(\psi)$. Therefore, in light of the argument in the previous paragraph, we only need to show $\eta = \infty$. But this follows since $f(R,1,\dots,1) =  Rf(1,R^{-1},\dots,R^{-1})> Rf(1,0,\dots,0)\rightarrow\infty$ as $R\rightarrow\infty$.
\end{proof}

\begin{proof}[Proof of Lemma \ref{600}]
	\textbf{Step 1:} In this step we show that if $w\in C^0$ is a viscosity supersolution to \eqref{1}, then it is a viscosity supersolution to \eqref{2}. Let $u\in C^2$ touch $w$ from below at $x_0$, so that 
	\begin{align}\label{127}
	f(\lambda_1^u,\lambda_2^u,\dots,\lambda_2^u)(x_0)\geq \frac{1}{2} \quad \text{and} \quad (\lambda_1^u,\lambda_2^u,\dots,\lambda_2^u)(x_0)\in\Gamma. 
	\end{align} 
	We wish to show that \eqref{126} holds. For this purpose, it suffices to show $\lambda_2^u(x_0)\in\operatorname{Dom}(\psi)$ -- once this is proved, the inequality $\lambda_1^u(x_0) \geq \psi(\lambda_2^u(x_0))$ follows immediately from the definition of $\psi$, \eqref{127} and monotonicity of $f$. By Proposition \ref{140}, either $\operatorname{Dom}(\psi) = (\frac{1}{2\eta},\infty)$ (in which case $(1,0,\dots,0)\in\partial\Gamma$) or $\operatorname{Dom}(\psi) = \mathbb{R}$ (in which case $\eta = \infty$). There is nothing to prove in the latter case, so let us consider the former case. Suppose for a contradiction that $\lambda_2^u(x_0)\not\in\operatorname{Dom}(\psi)$, i.e.~$\lambda_2^u(x_0) \leq \frac{1}{2\eta}$. If $\eta=\infty$, then the differential inclusion in \eqref{127} implies $(\lambda^1_u(x_0), 0, \dots,0)\in\Gamma$, which clearly contradicts $(1,0,\dots,0)\in\partial\Gamma$. If $\eta<\infty$, then by \eqref{127} and monotonicity we have
	\begin{align*}
	\frac{1}{2}\leq f(\lambda_1^u,\lambda_2^u,\dots,\lambda_2^u)(x_0) &  \leq f\bigg(\lambda_1^u(x_0), \frac{1}{2\eta}, \dots, \frac{1}{2\eta}\bigg) = \frac{1}{2\eta}f(2\eta\lambda_1^u(x_0),1,\dots,1),
	\end{align*}
	and hence $f(2\eta\lambda_1^u(x_0),1,\dots,1) \geq \eta$. This contradicts \eqref{139} and strict monotonicity of $f$.\medskip

	\noindent\textbf{Step 2:} In this step we show that if $w\in C^0$ is a viscosity subsolution to \eqref{1}, then it is a viscosity subsolution to \eqref{2}. Let $u\in C^2$ touch $w$ from above at $x_0$, so that either 
	\begin{align*}
	(\lambda_1^u,\lambda_2^u,\dots,\lambda_2^u)(x_0)\in\mathbb{R}^n\backslash\Gamma
	\end{align*}
	or 
	\begin{align*}
	f(\lambda_1^u,\lambda_2^u,\dots,\lambda_2^u)(x_0)\leq \frac{1}{2} \quad \text{and} \quad (\lambda_1^u,\lambda_2^u,\dots,\lambda_2^u)(x_0)\in\Gamma. 
	\end{align*}
	We wish to show that either \eqref{135} or \eqref{136} holds. If \eqref{135} holds we are done, so suppose on the contrary that $\lambda_2^u(x_0)\in\operatorname{Dom}(\psi)$. An easy argument using the definition of $\psi$ and monotonicity of $f$ then demonstrates that $\lambda_1^u(x_0) \leq \psi(\lambda_2^u(x_0))$, i.e.~\eqref{136} holds. \medskip

	The proof that a viscosity supersolution (resp.~viscosity subsolution) to \eqref{2} is a viscosity supersolution (resp.~subsolution) to \eqref{1} follows arguments similar to those above and is omitted. 
\end{proof}

\section{Some further results}\label{appb}

In this appendix we prove some consequences of our main results that may be of independent interest but are not used in the main body of the paper.
	
\subsection{A counterexample to the comparison principle on unbounded domains}

\begin{prop}\label{c1}
	Suppose $(f,\Gamma)$ satisfies \eqref{21'}--\eqref{25'} and $\mu_\Gamma^+ \leq 1$. Then there exist smooth positive functions $u$ and $w$ on $\mathbb{R}_+^n$ satisfying
	\begin{align*}
	\begin{cases}
	f(\lambda(-A_u)) = f(\lambda(-A_w)) = \frac{1}{2} & \text{in }\mathbb{R}_+^n \\
	w > u = 0 & \text{on }\partial\mathbb{R}_+^n
	\end{cases}
	\end{align*}
	but $w(x) < u(x)$ for some $x\in\mathbb{R}_+^n$. 
\end{prop}
\begin{proof}
	Let $u=w^{(a)}$ for some $a>0$, where $w^{(a)}$ are the solutions from Theorem \ref{B}. For each $b>0$, $y' \in\partial\mathbb{R}_+^n = \mathbb{R}^{n-1}$ and $R<b$, define on $\mathbb{R}^n\backslash B_R(y',-b)$ 
	\begin{align*}
	w_{(y', -b),R}(x) = \frac{|x-(y',-b)|^2 - R^2}{2R}.
	\end{align*}
	Note that $w_{(y', -b),R} > 0$ on $\partial\mathbb{R}_+^n$ and, since $g_{w_{(y', -b),R}}$ is the hyperbolic metric on $\mathbb{R}^n\backslash B_R(y',-b)$, it holds that $f(\lambda(-A_{w_{(y', -b),R}}))=\frac{1}{2}$. Suppose for a contradiction that $w_{(y', -b),R} \geq u$ in $\mathbb{R}_+^n$ for all $b>0$, $y' \in\partial\mathbb{R}_+^n$ and $R<b$. Then for given $y'\in\partial\mathbb{R}_+^n$ and any $x\in\mathbb{R}_+^n$ of the form $x=(y',x_n)\in\mathbb{R}^n_+$, we have
	\begin{align*}
	u(x) \leq w_{(y',-b),R}(x) = \frac{(x_n+b)^2 - R^2}{2R}.
	\end{align*}
	Taking $R\rightarrow b$ and subsequently $b\rightarrow+\infty$, it follows that $u(x) \leq x_n$ for any $x\in\mathbb{R}_+^n$ of the form $x=(y',x_n)\in\mathbb{R}^n_+$,. Since $y'\in\partial\mathbb{R}_+^n$ was arbitrary, it follows that $u(x) \leq x_n$ on $\mathbb{R}_+^n$. This contradicts the fifth property in Theorem \ref{B}, which tells us that $u(x) = w^{(a)}(x_n) > w^{(0)}(x_n) = x_n$.
\end{proof}

\begin{cor}
	Suppose $(f,\Gamma)$ satisfies \eqref{21'}--\eqref{25'} and $\mu_\Gamma^+ \leq 1$, and let $B$ denote the open unit ball in $\mathbb{R}^n$. Then there exist $w\in C^\infty(\overline{B})$ and $u\in C^{\infty}_{\operatorname{loc}}(\overline{B}\backslash\{p\})$ satisfying
	\begin{align*}
	\begin{cases}f(\lambda(-A_u)) = f(\lambda(-A_w)) = \frac{1}{2} & \text{in }B \\
	w > 0 & \text{on }\partial B \\
	u = 0 & \text{on }\partial B \backslash\{p\}
	\end{cases}
	\end{align*}
	but $w(x) < u(x)$ for some $x\in B$. 
\end{cor}
\begin{proof}
	Let $\phi$ be a M\"obius transformation on $\mathbb{R}^n\cup\{\infty\}$ such that $\phi|_B:B\rightarrow\mathbb{R}_+^n$ is a bijection, and let $v = w_{(y',-b),R}$ for some choice of $b>0$, $y' \in\partial\mathbb{R}_+^n$ and $R<b$ satisfying the conclusion of Corollary \ref{c1} (with $w=w_{(y',-b),R}$ therein). Then by conformal invariance (and using the notation introduced above \eqref{15'}), 
	\begin{align*}
	\begin{cases}
	f(\lambda(-A_{u_\phi})) = f(\lambda(-A_{v_\phi})) = \frac{1}{2} & \text{in }B \\
	v_\phi > u_\phi = 0 & \text{on }\partial B\backslash\phi^{-1}(\infty),
	\end{cases}
	\end{align*}
	but $v_\phi(x) < u_\phi(x)$ for some $x\in B$. Note that since $v$ is bounded away from infinity, $v_\phi$ extends uniquely to a function $w$ on $\overline{B}$ and this extension is positive on $\partial B$. 
\end{proof}

\subsection{Some non-existence results for admissible metrics with prescribed boundary data}

As discussed in the introduction, a large part of our solution to the fully nonlinear Loewner-Nirenberg problem when $\mu_\Gamma^+>1$ in \cite{DN23} consisted of obtaining local $C^0$ boundary estimates on solutions. Theorem \ref{D} demonstrates that such estimates in general fail when $\mu_\Gamma^+ \leq 1$. The construction of the barrier function in \cite[Proposition 1.1]{DN23} is therefore not possible when $\mu_\Gamma^+ \leq 1$, as seen in the following proposition. To give a more general statement, we note that the notion of a continuous viscosity supersolution given in Definition \ref{137} extends immediately to the notion of a lower semicontinuous viscosity subsolution $w_1$ on $\Omega$, i.e.~functions $w_1:\Omega\rightarrow(0,+\infty]$, $w_1\not\equiv+\infty$, satisfying $\liminf_{x\rightarrow y}w_1(x)\geq w_1(y)$ for all $y\in\Omega$.

\begin{prop}\label{505}
	Suppose $(f,\Gamma)$ satisfies \eqref{21'}--\eqref{25'} and $\mu_\Gamma^+ \leq 1$. Let $\Omega$ be a (possibly unbounded) domain with smooth non-empty boundary and suppose $\Sigma\subset\partial\Omega$ has non-empty interior with respect to the topology of the boundary. Then there is no lower semicontinuous function $w>0$ on $\Omega$ such that both $w(x)\rightarrow+\infty$ as $d(x,\Sigma)\rightarrow0$ and
	\begin{align}\label{210}
	f(\lambda(-A_w)) \geq \ep >0, \quad \lambda(-A_w)\in\Gamma \quad \text{in the viscosity sense on }\Omega
	\end{align}
	for some $\ep>0$. 
\end{prop}
\begin{proof}
	Suppose for a contradiction that such a metric $g_w$ exists. After rescaling $w$, we may assume without loss of generality that $\ep = \frac{1}{2}$. Choose a point $x_0$ in the interior of $\Sigma$; without loss of generality, we may assume $x_0$ is the origin and $\Sigma$ is tangent to the hyperplane $\{x_n=0\}$ at $x_0$. Furthermore, by conformal invariance of \eqref{210} and by inverting through a sphere $S$ contained in $\mathbb{R}^n\backslash\Omega$ and tangent to $\Sigma$ at the origin if necessary, we may assume that $\Omega$ is strictly convex near the origin and lies below the hyperplane $\{x_n=0\}$.\medskip 
	
	Now fix $t_0>0$ sufficiently small so that $\Omega\cap \{x_n>-t_0\}$ contains a connected component $\Omega_{t_0}$ whose boundary is Lipschitz and consists of a piece lying within $\Sigma$ and a piece lying on the hyperplane $\{x_n=-t_0\}$. Let $w_{t_0}^{(a)}(x) = w^{(a)}(x_n+t_0)$ for $x_n \geq -t_0$, where $w^{(a)}$ are the solutions from Theorem \ref{B}. Then by the comparison principle in Proposition \ref{800} (which also applies for semicontinuous functions, see \cite{LNW18}), $w > w^{(a)}_{t_0}$ in $\Omega_{t_0}$ for each $a$. But by Theorem \ref{D}, $w^{(a)}_{t_0}\rightarrow\infty$ pointwise in $\Omega_{t_0}$ as $a\rightarrow\infty$, which yields a contradiction for sufficiently large $a$.
\end{proof}

\end{appendices}

	\addtocontents{toc}{\protect\setcounter{tocdepth}{1}}

\footnotesize
\bibliography{references}{}

\begin{thebibliography}{10}

\bibitem{A58}
{\sc A.~D. Aleksandrov}, {\em Uniqueness theorems for surfaces in the large.
  {V}}, Vestnik Leningrad. Univ., 13 (1958), pp.~5--8.

\bibitem{AM88}
{\sc P.~Aviles and R.~C. McOwen}, {\em Complete conformal metrics with negative
  scalar curvature in compact {R}iemannian manifolds}, Duke Math. J., 56
  (1988), pp.~395--398.

\bibitem{BN91}
{\sc H.~Berestycki and L.~Nirenberg}, {\em On the method of moving planes and
  the sliding method}, Bol. Soc. Brasil. Mat. (N.S.), 22 (1991), pp.~1--37.

\bibitem{CNS3}
{\sc L.~Caffarelli, L.~Nirenberg, and J.~Spruck}, {\em The {D}irichlet problem
  for nonlinear second-order elliptic equations. {III}. {F}unctions of the
  eigenvalues of the {H}essian}, Acta Math., 155 (1985), pp.~261--301.

\bibitem{CC95}
{\sc L.~A. Caffarelli and X.~Cabr\'{e}}, {\em Fully nonlinear elliptic
  equations}, vol.~43 of American Mathematical Society Colloquium Publications,
  American Mathematical Society, Providence, RI, 1995.

\bibitem{CGS89}
{\sc L.~A. Caffarelli, B.~Gidas, and J.~Spruck}, {\em Asymptotic symmetry and
  local behavior of semilinear elliptic equations with critical {S}obolev
  growth}, Comm. Pure Appl. Math., 42 (1989), pp.~271--297.

\bibitem{CW18}
{\sc J.~S. Case and Y.~Wang}, {\em Boundary operators associated to the
  {$\sigma_k$}-curvature}, Adv. Math., 337 (2018), pp.~83--106.

\bibitem{CGY02b}
{\sc S.-Y.~A. Chang, M.~J. Gursky, and P.~C. Yang}, {\em An a priori estimate
  for a fully nonlinear equation on four-manifolds}, J. Anal. Math., 87 (2002),
  pp.~151--186.

\bibitem{CGY03}
\leavevmode\vrule height 2pt depth -1.6pt width 23pt, {\em Entire solutions of
  a fully nonlinear equation}, in Lectures on partial differential equations,
  vol.~2 of New Stud. Adv. Math., Int. Press, Somerville, MA, 2003, pp.~43--60.

\bibitem{CC91}
{\sc W.~X. Chen and C.~Li}, {\em Classification of solutions of some nonlinear
  elliptic equations}, Duke Math. J., 63 (1991), pp.~615--622.

\bibitem{CSF96}
{\sc M.~Chipot, I.~Shafrir, and M.~Fila}, {\em On the solutions to some
  elliptic equations with nonlinear {N}eumann boundary conditions}, Adv.
  Differential Equations, 1 (1996), pp.~91--110.

\bibitem{CLL23}
{\sc B.~Z. Chu, Y.~Y. Li, and Z.~Li}, {\em Liouville theorems for conformally
  invariant fully nonlinear equations {I}}, arXiv:2311.07542 [math.AP],
  (2023).

\bibitem{CLL24b}
\leavevmode\vrule height 2pt depth -1.6pt width 23pt, {\em Liouville {T}heorem
  with {B}oundary {C}onditions from {C}hern--{G}auss--{B}onnet {F}ormula},
  https://arxiv.org/abs/2410.16384,  (2024).

\bibitem{CLL24a}
\leavevmode\vrule height 2pt depth -1.6pt width 23pt, {\em On the fully
  nonlinear {Y}amabe problem with constant boundary mean curvature. {I}},
  https://arxiv.org/abs/2410.09683,  (2024).

\bibitem{DN23}
{\sc J.~A.~J. Duncan and L.~Nguyen}, {\em The $\sigma_k$-{L}oewner-{N}irenberg
  problem on {R}iemannian manifolds for $k<\frac{n}{2}$},
  arxiv.org/abs/2310.01346, to appear in Anal.~PDE,  (2023).

\bibitem{DN25b}
\leavevmode\vrule height 2pt depth -1.6pt width 23pt, {\em The
  $\sigma_k$-{L}oewner-{N}irenberg problem on {R}iemannian manifolds for
  $k=\frac{n}{2}$ and beyond},  (2025).

\bibitem{E90}
{\sc J.~F. Escobar}, {\em Uniqueness theorems on conformal deformation of
  metrics, {S}obolev inequalities, and an eigenvalue estimate}, Comm. Pure
  Appl. Math., 43 (1990), pp.~857--883.

\bibitem{FMW23}
{\sc H.~Fang, B.~Ma, and W.~Wei}, {\em A {L}iouville's theorem for some
  {M}onge-{A}mp\`ere type equations}, J. Funct. Anal., 285 (2023), pp.~Paper
  No. 109973, 47.

\bibitem{GW06}
{\sc Y.~Ge and G.~Wang}, {\em On a fully nonlinear {Y}amabe problem}, Ann. Sci.
  \'{E}cole Norm. Sup. (4), 39 (2006), pp.~569--598.

\bibitem{GNN79}
{\sc B.~Gidas, W.~M. Ni, and L.~Nirenberg}, {\em Symmetry and related
  properties via the maximum principle}, Comm. Math. Phys., 68 (1979),
  pp.~209--243.

\bibitem{GNN81}
\leavevmode\vrule height 2pt depth -1.6pt width 23pt, {\em Symmetry of positive
  solutions of nonlinear elliptic equations in {${\bf R}^{n}$}}, in
  Mathematical analysis and applications, {P}art {A}, vol.~7 of Adv. Math.
  Suppl. Stud., Academic Press, New York-London, 1981, pp.~369--402.

\bibitem{GLN18}
{\sc M.~d.~M. Gonz\'{a}lez, Y.~Y. Li, and L.~Nguyen}, {\em Existence and
  uniqueness to a fully nonlinear version of the {L}oewner-{N}irenberg
  problem}, Commun. Math. Stat., 6 (2018), pp.~269--288.

\bibitem{GSW11}
{\sc M.~J. Gursky, J.~Streets, and M.~Warren}, {\em Existence of complete
  conformal metrics of negative {R}icci curvature on manifolds with boundary},
  Calc. Var. PDE, 41 (2011), pp.~21--43.

\bibitem{GV07}
{\sc M.~J. Gursky and J.~A. Viaclovsky}, {\em Prescribing symmetric functions
  of the eigenvalues of the {R}icci tensor}, Ann. of Math. (2), 166 (2007),
  pp.~475--531.

\bibitem{HS20}
{\sc Q.~Han and W.~Shen}, {\em The {L}oewner-{N}irenberg problem in singular
  domains}, J. Funct. Anal., 279 (2020), pp.~108604, 43.

\bibitem{JLX05}
{\sc Q.~Jin, Y.~Y. Li, and H.~Xu}, {\em Nonexistence of positive solutions for
  some fully nonlinear elliptic equations}, Methods Appl. Anal., 12 (2005),
  pp.~441--449.

\bibitem{LL03}
{\sc A.~Li and Y.~Y. Li}, {\em On some conformally invariant fully nonlinear
  equations}, Comm. Pure Appl. Math., 56 (2003), pp.~1416--1464.

\bibitem{LL05}
\leavevmode\vrule height 2pt depth -1.6pt width 23pt, {\em On some conformally
  invariant fully nonlinear equations. {II}. {L}iouville, {H}arnack and
  {Y}amabe}, Acta Math., 195 (2005), pp.~117--154.

\bibitem{LL06}
\leavevmode\vrule height 2pt depth -1.6pt width 23pt, {\em A fully nonlinear
  version of the {Y}amabe problem on manifolds with boundary}, J. Eur. Math.
  Soc. (JEMS), 8 (2006), pp.~295--316.

\bibitem{Li07}
{\sc Y.~Y. Li}, {\em Degenerate conformally invariant fully nonlinear elliptic
  equations}, Arch. Ration. Mech. Anal., 186 (2007), pp.~25--51.

\bibitem{Li09}
\leavevmode\vrule height 2pt depth -1.6pt width 23pt, {\em Local gradient
  estimates of solutions to some conformally invariant fully nonlinear
  equations}, Comm. Pure Appl. Math., 62 (2009), pp.~1293--1326.

\bibitem{LLL21}
{\sc Y.~Y. Li, H.~Lu, and S.~Lu}, {\em A {L}iouville {T}heorem for {M}\"obius
  {I}nvariant {E}quations}, Peking Mathematical Journal (2021).

\bibitem{LN09b}
{\sc Y.~Y. Li and L.~Nguyen}, {\em A fully nonlinear version of the {Y}amabe
  problem on locally conformally flat manifolds with umbilic boundary},
  https://arxiv.org/abs/0911.3366,  (2009).

\bibitem{LN14}
\leavevmode\vrule height 2pt depth -1.6pt width 23pt, {\em A compactness
  theorem for a fully nonlinear {Y}amabe problem under a lower {R}icci
  curvature bound}, J. Funct. Anal., 266 (2014), pp.~3741--3771.

\bibitem{LN14b}
\leavevmode\vrule height 2pt depth -1.6pt width 23pt, {\em Harnack inequalities
  and {B}\^{o}cher-type theorems for conformally invariant, fully nonlinear
  degenerate elliptic equations}, Comm. Pure Appl. Math., 67 (2014),
  pp.~1843--1876.

\bibitem{LNW18}
{\sc Y.~Y. Li, L.~Nguyen, and B.~Wang}, {\em Comparison principles and
  {L}ipschitz regularity for some nonlinear degenerate elliptic equations},
  Calc. Var. Partial Differential Equations, 57 (2018), pp.~Paper No. 96, 29.

\bibitem{LNX22}
{\sc Y.~Y. Li, L.~Nguyen, and J.~Xiong}, {\em Regularity of viscosity solutions
  of the {$\sigma_k$}-{L}oewner-{N}irenberg problem}, Proc. Lond. Math. Soc.
  (3), 127 (2023), pp.~1--34.

\bibitem{LZ03}
{\sc Y.~Y. Li and L.~Zhang}, {\em Liouville-type theorems and {H}arnack-type
  inequalities for semilinear elliptic equations}, J. Anal. Math., 90 (2003),
  pp.~27--87.

\bibitem{LZ95}
{\sc Y.~Y. Li and M.~Zhu}, {\em Uniqueness theorems through the method of
  moving spheres}, Duke Math. J., 80 (1995), pp.~383--417.

\bibitem{Liou1853}
{\sc J.~Liouville}, {\em Sur l'\'{e}quation aux diff\'{e}rences partielles
  $\frac{d^2\log\lambda}{dudv}\pm \frac{\lambda}{2a^2}=0$.}, Journal de
  Math\'{e}matiques Pures et Appliqu\'{e}es,  (1853), pp.~71--72.

\bibitem{LN74}
{\sc C.~Loewner and L.~Nirenberg}, {\em Partial differential equations
  invariant under conformal or projective transformations}, in Contributions to
  analysis (a collection of papers dedicated to {L}ipman {B}ers), Academic
  Press, New York-London, 1974, pp.~245--272.

\bibitem{LouZhu99}
{\sc Y.~Lou and M.~Zhu}, {\em Classifications of nonnegative solutions to some
  elliptic problems}, Differential Integral Equations, 12 (1999), pp.~601--612.

\bibitem{Ob71}
{\sc M.~Obata}, {\em The conjectures on conformal transformations of
  {R}iemannian manifolds}, J. Differential Geometry, 6 (1971/72), pp.~247--258.

\bibitem{Oss57}
{\sc R.~Osserman}, {\em On the inequality {$\Delta u\geq f(u)$}}, Pacific J.
  Math., 7 (1957), pp.~1641--1647.

\bibitem{Sav07}
{\sc O.~Savin}, {\em Small perturbation solutions for elliptic equations},
  Comm. Partial Differential Equations, 32 (2007), pp.~557--578.

\bibitem{SY88}
{\sc R.~M. Schoen and S.-T. Yau}, {\em Conformally flat manifolds, {K}leinian
  groups and scalar curvature}, Invent. Math., 92 (1988), pp.~47--71.

\bibitem{Ser71}
{\sc J.~Serrin}, {\em A symmetry problem in potential theory}, Arch. Rational
  Mech. Anal., 43 (1971), pp.~304--318.

\bibitem{TWZ22}
{\sc S.~Tang, L.~Wang, and M.~Zhu}, {\em Nonlinear elliptic equations on the
  upper half space}, Commun. Contemp. Math., 24 (2022), pp.~Paper No. 2050085,
  20.

\bibitem{Via00b}
{\sc J.~A. Viaclovsky}, {\em Conformally invariant {M}onge-{A}mp\`ere
  equations: global solutions}, Trans. Amer. Math. Soc., 352 (2000),
  pp.~4371--4379.

\bibitem{Wei24}
{\sc W.~Wei}, {\em Liouville {T}heorem for $k$-curvature equation with fully
  nonlinear boundary in half space}, https://arxiv.org/abs/2403.19268,  (2024).

\bibitem{Yuan22}
{\sc R.~Yuan}, {\em The partial uniform ellipticity and prescribed problems on
  the conformal classes of complete metrics}, https://arxiv.org/abs/2203.13212,
   (2022).

\bibitem{Yuan24}
\leavevmode\vrule height 2pt depth -1.6pt width 23pt, {\em Notes on conformal
  metrics of negative curvature on manifolds with boundary}, J. Math. Study, 57
  (2024), pp.~373--378.

\end{thebibliography}
\bibliographystyle{siam}
\end{document}